\input ./style/arxiv-general.cfg
\documentclass[aop,MSNbibl,seceqn,dvips]{arximspdf}
\makeatletter
   \@ifpackageloaded{graphicx}{}{\usepackage{graphicx}}
\makeatother
\usepackage{mathbh}
\usepackage{stmaryrd}

\doi{10.1214/14-AOP970} % kopijuoti is laisko
\volume{44}
\issue{1}
\pubyear{2016}
\firstpage{171}
\lastpage{205}
\docsubty{FLA}

\makeatletter
\newcommand{\rrvert}{\vert}
\newcommand{\llvert}{\vert}
\newcommand{\eqref}[1]{(\ref{#1})}
\newtheorem{theorem}{Theorem}[section]
\newtheorem{corollary}[theorem]{Corollary}
\newtheorem{lemma}[theorem]{Lemma}
\newtheorem{proposition}[theorem]{Proposition}
\newproclaim{definition}[theorem]{Definition}
\newproclaim{example}[theorem]{Example}
\newproclaim{remark}[theorem]{Remark}
\makeatother

\begin{document}
\begin{frontmatter}

\title{Strong supermartingales and limits of nonnegative~martingales}
\runtitle{Supermartingales and limits of nonnegative
martingales}

\begin{aug}
\author[A]{\fnms{Christoph}~\snm{Czichowsky}\corref{}\ead[label=e1]{c.czichowsky@lse.ac.uk}\thanksref{TT1}}%,
%\author[]{\fnms{}~\snm{}\ead[label=]{}}
\and
\author[B]{\fnms{Walter}~\snm{Schachermayer}\ead[label=e2]{walter.schachermayer@univie.ac.at}\thanksref{T1}}
\runauthor{C. Czichowsky and W. Schachermayer}
\affiliation{London School of Economics and Political Science and Universit\"at Wien}
%\dedicated{}
\address[A]{Department of Mathematics\\
London School of Economics \\
\quad and Political Science\\
Columbia House, Houghton Street\\
London WC2A 2AE\\
United Kingdom\\
\printead{e1}}
\address[B]{Fakult\"at f\"ur Mathematik\\
Universit\"at Wien\\
Oskar-Morgenstern-Platz 1\\
A-1090 Wien\\
Austria\\
\printead{e2}}
\end{aug}
\thankstext{TT1}{Supported by the European Research Council (ERC) under Grant FA506041.}
\thankstext{T1}{Supported in part by the Austrian
Science Fund (FWF) under Grant P25815, the European Research Council
(ERC) under Grant FA506041 and by the Vienna Science and Technology
Fund (WWTF) under Grant MA09-003.}

% HISTORY:
\received{\smonth{12} \syear{2013}}
\revised{\smonth{7} \syear{2014}}
%\accepted{\smonth{} \syear{}}

% ABSTRACT
\begin{abstract}
Given a sequence $(M^n)^\infty_{n=1}$ of nonnegative
martingales starting at \mbox{$M^n_0=1$}, we find a sequence of convex
combinations $(\widetilde{M}^n)^\infty_{n=1}$ and a limiting process
$X$ such that $(\widetilde{M}^n_\tau)^\infty_{n=1}$ converges in
probability to $X_\tau$, for all finite stopping times $\tau$. The
limiting process $X$ then is an optional strong supermartingale. A
counterexample reveals that the convergence in probability cannot be
replaced by almost sure convergence in this statement. We also give
similar convergence results for sequences of optional strong
supermartingales $(X^n)^\infty_{n=1}$, their left limits
$(X^n_-)^\infty_{n=1}$ and their stochastic integrals $(\int\varphi \,dX^n)^\infty
_{n=1}$ and explain the relation to the notion of the Fatou limit.
\end{abstract}

% KEYWORDS
\end{frontmatter}

%s1 #&#
\section{Introduction}

Koml\'os's lemma (see \cite{K67,S86} and \cite{DS94}) is a
classical result on the convergence of random variables that can be
used as a substitute for compactness. It has turned out to be very
useful, similarly to the Bolzano--Weierstrass theorem, and has become a
work horse of stochastic analysis in the past decades. In this paper,
we generalize this result to work directly with nonnegative
martingales and convergence in probability simultaneously at all finite
stopping times.

Let us briefly explain this in more detail. Koml\'os's subsequence
theorem states that given a bounded sequence $(f_n)^\infty_{n=1}$ of
random variables in $L^1(P)$, there exists a random variable
$f \in L^1(P)$ and a subsequence $(f_{n_k})^\infty_{k=1}$ such that the Ces\`{a}ro means of any subsubsequence $(f_{n_{k_j}})^\infty_{j=1}$ converge
almost surely to $f$. It quickly follows that there exists a sequence
$(\tilde{f}_n)^\infty_{n=1}$ of convex combinations $\tilde{f}_n\in
\operatorname{conv}(f_n, f_{n+1}, \ldots)$ that converges to $f$
almost surely that
we refer to as Koml\'os's lemma.

Replacing the almost sure convergence by the concept of \textit{Fatou
convergence}, F\"ollmer and Kramkov \cite{FK97} obtained the following
variant of Koml\'os's lemma for stochastic processes. Given a sequence
$(M^n)^\infty_{n=1}$ of nonnegative martingales $M^n=(M^n_t)_{0 \leq t
\leq1}$ starting at $M^n_0=1$, there exists a sequence $(\overline{M}^n)^\infty
_{n=1}$ of convex\vadjust{\goodbreak} combinations $\overline{M}^n \in\operatorname{conv}(M^n, M^{n+1}, \ldots)$
and a nonnegative c\`adl\`ag supermartingale $\overline{X}=(\overline
{X}_t)_{0 \leq t
\leq1}$ starting at $X_0=1$ such that $\overline{M}^n$ is Fatou convergent
along the rationals $\mathbb{Q}\cap[0,1]$ to $\overline{X}$ in the
sense that
\[
\overline{X}_t = \mathop{\overline{\lim}}_{q \in\mathbb{Q}\cap[0,1], q
\downarrow t} \mathop{
\overline{\lim}}_{n\to\infty} \overline{M}^n_q=\mathop{
\underline{\lim}}_{q \in\mathbb{Q}\cap
[0,1], q \downarrow t} \mathop{\underline{\lim}}_{n\to\infty}
\overline{M}^n_q, \qquad\mbox{$P$-a.s.},
\]
for all $t \in[0,1)$ and $\overline{X}_1=\lim_{n\to\infty}
\overline{M}^n_1$.

In this paper, we are interested in a different version of Koml\'os's
lemma for nonnegative martingales in the following sense. Given the
sequence $(M^n)^\infty_{n=1}$ of nonnegative martingales as above and
a finite stopping time $\tau$ defining $f_n:=M^n_\tau$ gives a sequence
of nonnegative random variables that is bounded in $L^1(P)$. By
Koml\'os's lemma there exist convex combinations $\widetilde{M}^n\in\operatorname{conv}(M^n,M^{n+1}, \ldots)$ such that $\widetilde{M}^n_\tau$ converges in
probability to
some random variable $f_\tau$. The question is then, if we can find
\emph{one} sequence $(\widetilde{M}^n)^\infty_{n=1}$ of convex
combinations $\widetilde{M}^n \in\operatorname{conv}(M^n, M^{n+1},
\ldots)$ and a\vspace*{1pt}
stochastic process $X=(X_t)_{0 \leq t \leq1}$ such that we have that
$\widetilde{M}^n_\tau$ converges to $X_\tau$ in probability for
\emph{all} finite stopping times $\tau$.

Our first main result (Theorem~\ref{c1}) shows that this is possible
and that the limiting process $X=(X_t)_{0 \leq t \leq1}$ is an \emph{optional strong supermartingale}. These supermartingales have been
introduced by Mertens~\cite{M72} and are optional processes that
satisfy the supermartingale inequality for all finite stopping times.
This indicates that optional strong supermartingales are the natural
processes for our purpose to work with, and we expand in Theorem~\ref{t1} our convergence result from martingales $(M^n)_{n=1}^\infty$ to
optional strong supermartingales $(X^n)_{n=1}^\infty$.

In dynamic optimization problems our results can be used as substitute
for compactness; compare, for example, \cite{DS99,FK97,KS01,KZ04,S04}. Here the martingales $M^n$ are usually
a minimizing sequence of density processes of equivalent martingale
measures for the dual problem or, as in \cite{DS99} and \cite{FK97},
the wealth processes of self-financing trading strategies.

At a fixed stopping time the convergence in probability can always be
strengthened to almost sure convergence by simply passing to a
subsequence. By means of a counterexample (Proposition~\ref{Ex2}) we
show that this is not possible for all stopping times simultaneously.

Conversely, one can ask what the smallest class of stochastic processes
is that is closed under convergence in probability at all finite
stopping times and contains all bounded martingales. Our second
contribution (Theorem~\ref{t2}) is to show that this is precisely the
class of all optional strong supermartingales provided the underlying
probability space is sufficiently rich to support a Brownian motion.

As the limiting strong supermartingale of a sequence of martingales in
the sense of convergence in probability at all finite stopping times is
no longer a semimartingale, we need to restrict the integrands to be
predictable finite variation processes $\varphi=(\varphi_t)_{0 \leq t
\leq1}$ to come up with a similar convergence result for stochastic
integrals in Proposition~\ref{pSI}. For this, we need to extend our
convergence result to ensure the convergence of the left limit
processes $(X^n_-)_{n=1}^\infty$ in probability at all finite stopping
times to a limiting process $X^{(0)}=(X^{(0)})_{0\leq t\leq1}$ as well
after possibly passing once more to convex combinations. It turns out
that $X^{(0)}$ is a \emph{predictable strong supermartingale} that
does, in general, \emph{not} coincide with the left limit process $X_-$
of the limiting optional strong supermartingale $X$. The notion of a
predictable strong supermartingale has been introduced by Chung and
Glover \cite{CG79} and refers to predictable processes that satisfy the
supermartingale inequality for all \emph{predictable} stopping times.
Using instead of the time interval $I=[0,1]$ its \emph{Alexandroff
double arrow space} $\widetilde{I}=[0,1]\times\{0,1\}$ as index set we
can merge both limiting strong supermartingales into one
supermartingale $X=(X_{\tilde{t}})_{\tilde{t} \in\widetilde{I}}$
indexed by $\widetilde{I}$.

Our motivation for studying these questions comes from portfolio
optimization under transaction costs in mathematical finance in \cite{CS14}. While for the problem without transaction costs the solution to
the dual problem is always attained as a Fatou limit, the dual
optimizer under transaction costs is in general a truly l\`adl\`ag
optional strong supermartingale. So we expect our results naturally to
appear whenever one is optimizing over nonnegative martingales that
are not uniformly integrable or stable under concatenation, and they
might find other applications as well.

The paper is organized as follows. We formulate the problem and state
our main results in Section~\ref{sec2}. The proofs are given in
Sections~\ref{sec3}, \ref{sec5}, \ref{sec6} and \ref{sec7}. Section~\ref{sec4} provides the counterexample that our convergence results
cannot be strengthened to almost sure convergence.

%s2 #&#
\section{Formulation of the problem and main results}\label{sec2}
Let $(\Omega, \mathcal{F}, P)$ be a~probability space and
$L^0(P)=L^0(\Omega,
\mathcal{F}, P)$ the space of all real-valued random variables. As
usual we equip $L^0(P)$ with the topology of convergence in probability
and denote by $L^0_+(P)=L^0(\Omega, \mathcal{F}, P;\mathbb{R}_+)$ its
positive cone. We call a subset $A$ of $L^0(P)$ bounded in probability
or simply bounded in $L^0(P)$, if $\lim_{m\to\infty} \sup_{f
\in
A} P (|f|>m)=0$.

Koml\'os's subsequence theorem (see \cite{K67} and \cite{S86}) states
the following.

%th2.1 #&#
\begin{theorem}\label{kssthm}
Let $(f_n)^\infty_{n=1}$ be a bounded sequence of random variables in
$L^1(\Omega, \mathcal{F}, P)$. Then there exists a subsequence
$(f_{n_k})^\infty_{k=1}$ and a random variable $f$ such that the Ces\`{a}ro means $\frac{1}{J}\sum_{j=1}^J f_{n_{k_j}}$ of any subsubsequence
$(f_{n_{k_j}})^\infty_{j=1}$ converge $P$-almost surely to $f$, as
$J\to
\infty$.
\end{theorem}

In applications this result is often used in the following variant that
we also refer to as Koml\'os's lemma; compare Lemma~A.1 in \cite{DS94}.

%co2.2 #&#
\begin{corollary}\label{kl}
Let $(f_n)^\infty_{n=1}$ be a sequence of nonnegative random variables
that is bounded in $L^1(P)$. Then there exists a sequence $(\tilde
{f}_n)^\infty_{n=1}$ of convex combinations
\[
\tilde{f}_n \in\operatorname{conv}(f_n,
f_{n+1}, \ldots)
\]
and a nonnegative\vadjust{\goodbreak} random variable $f \in L^1 (P)$ such that
$\tilde{f}_n \stackrel{P\mbox{-}\mathit{a.s.}}{\longrightarrow} f$.
\end{corollary}

As has been illustrated by the work of Kramkov and Schachermayer \cite{KS01} and \v Zitkovi\'c \cite{Z10} (see also \cite{S04}) Koml\'os's
lemma can be used as a substitute for compactness, for example, in the
derivation of minimax theorems for Lagrange functions, where the
optimization is typically over convex sets.
Replacing the $P$-almost sure convergence by the concept of \emph{Fatou
convergence} F\"ollmer and Kramkov \cite{FK97} used Koml\'os's lemma to
come up with a similar convergence result for stochastic processes.
For this, we equip the probability space $(\Omega, \mathcal{F}, P)$
with a
filtration $\mathbb{F}=(\mathcal{F}_t)_{0 \leq t \leq1}$ satisfying\vspace*{1pt}
the usual
conditions of right continuity and completeness and let $(M^n)^\infty
_{n=1}$ be a sequence of nonnegative martingales $M^n=(M^n_t)_{0 \leq
t \leq1}$ starting at $M^n_0=1$. For all unexplained notation from the
general theory of stochastic processes and stochastic integration, we
refer to the book of Dellacherie and Meyer~\cite{DM82}.

The construction of the Fatou limit by F\"ollmer and Kramkov can be
summarized as in the following proposition.

%pr2.3 #&#
\begin{proposition}[(Lemma~5.2 of \cite{FK97})]\label{pFatou}
Let\vspace*{1pt} $(M^n)^\infty_{n=1}$ be a sequence of nonnegative martingales
$M^n=(M^n_t)_{0 \leq t \leq1}$ starting at $M^n_0=1$. Then there
exists a~sequence $(\overline{M}^n)^\infty_{n=1}$ of convex combinations
\[
\overline{M}^n \in\operatorname{conv}\bigl(M^n,
M^{n+1}, \ldots\bigr)
\]
and nonnegative random variables $Z_q$ for $q \in\mathbb{Q} \cap
[0,1]$ such that:
\begin{longlist}[(3)]
\item[(1)] $\overline{M}^n_q \stackrel{P\mbox{-}\mathit{a.s.}}{\longrightarrow} Z_q$ for
all $q \in\mathbb{Q}\cap[0,1]$;
\item[(2)] the process $\overline{X}=(\overline{X}_t)_{0 \leq
t \leq1}$ given by
%
%e2.1 #&#
\begin{equation}\label{defFatou}
\overline{X}_t:=\lim_{q \in\mathbb{Q} \cap[0,1],  q
\downarrow t}
Z_q\quad\mbox{and}\quad  \overline{X}_1=Z_1
\end{equation}
is a c\`adl\`ag supermartingale;
\item[(3)] the process $\overline{X}=(\overline{X}_t)_{0 \leq
t \leq1}$ is the \emph{Fatou limit} of the sequence $(\overline{M}^n)^\infty_{n=1}$ along
$\mathbb{Q}
\cap[0,1]$, that is,
\begin{eqnarray*}
\overline{X}_t & = &\mathop{\overline{\lim}}_{q \in\mathbb{Q} \cap[0,1],  q
\downarrow t} \mathop{
\overline{\lim}}_{n \to\infty} \overline{M}^n_q =
\mathop{\underline{\lim}}_{q \in
\mathbb{Q} \cap
[0,1],  q \downarrow t}\mathop{\underline{\lim}}_{n \to\infty}
\overline{M}^n_q, \qquad \mbox{$P$-a.s.},\quad \mbox{and}\\
\overline{X}_1 &=& \lim_{n\to
\infty}\overline{M}^n_1.
\end{eqnarray*}
\end{longlist}
\end{proposition}

Here\vspace*{1pt} it is important to note that $\lim_{q\in\mathbb{Q} \cap[0,1],
q\downarrow t}$ denotes the limit to $t$ through all $q \in\mathbb{Q}
\cap[0,1]$ that are \emph{strictly} bigger than $t$. Therefore we do
not have in general that $\overline{X}_t=\lim_{n\to\infty}\overline
{M}^n_t$ for $t\in
[0,1)$, not even for $t\in\mathbb{Q} \cap[0,1]$, as is illustrated in
the simple example below.

%ex2.4 #&#
\begin{example}\label{Ex0}
Let $(Y_n)^\infty_{n=1}$ be a sequence of random variables taking
values in $\{0,n\}$ such that $P[Y_n=n]=\frac{1}{n}$ and define a
sequence $(M^n)^\infty_{n=1}$ of martingales $M^n=(M^n_t)_{0 \leq t
\leq1}$ by
\[
M^n_t= 1 + \bigl(Y^n-1\bigr)
\mathbh{1}_{\rrbracket {1}/{2}(1+{1}/{n}),
1\rrbracket}(t).
\]
Then $M^n_t$ converges to $\mathbh{1}_{\llbracket0,{1}/{2}\rrbracket} (t)$ for each $t \in[0,1]$. However, the c\`adl\`ag
Fatou limit is $\overline{X}=\mathbh{1}_{\llbracket0,{1}/{2}\llbracket}(t)$.
\end{example}

The convergence, of course, also fails at stopping times in general.
This motivates us to ask for a different extension of Koml\'os's lemma
to nonnegative martingales in the following sense. Let $(M^n)^\infty
_{n=1}$ be again a sequence of nonnegative martingales $M^n=(M^n_t)_{0
\leq t \leq1}$ starting at $M^n_0=1$ and $\tau$ a finite stopping
time. Then defining $f_n:=M^n_\tau$ gives a sequence $(f_n)^\infty
_{n=1}$ of nonnegative random variables that are bounded in $L^1(P)$.
By Koml\'os's lemma there exist convex combinations $\widetilde{M}^n
\in\operatorname{conv}(M^n, M^{n+1}, \dots)$ and a nonnegative
random variable
$f_\tau$ such that
\[
\widetilde{M}^n_\tau=:\tilde{f}_n \mathop{\longrightarrow}^{P\mbox{-}\mathrm{a.s.}} f_\tau.
\]
The questions are then:
\begin{longlist}[(3)]
\item[(1)] Can we find \emph{one} sequence $(\widetilde{M}^n)^\infty_{n=1}$ of convex combinations
\[
\widetilde{M}^n \in\operatorname{conv}\bigl(M^n,
M^{n+1}, \ldots\bigr)
\]
such that, for \emph{all} finite stopping times $\tau$, we have
%
%e2.2 #&#
\begin{equation}
\label{q1}
\widetilde{M}^n_\tau\mathop{\longrightarrow}^{P\mbox{-}\mathrm{a.s.}} f_\tau
\end{equation}
for some random variables $f_\tau$ that may depend on the stopping
times $\tau$?
\item[(2)] If (1) is possible, can we find a stochastic
process $X=(X_t)_{0 \leq t \leq1}$ such that $X_\tau=f_\tau$ for all
finite stopping times $\tau$?
\item[(3)] If such a process $X=(X_t)_{0 \leq t \leq1}$ as in
(2) exists, what kind of process is it?
\end{longlist}

Let us start with the last question. If such a process $X=(X_t)_{0 \leq
t \leq1}$ exists, it follows from Fatou's lemma that it is (up to
optional measurability) an optional strong supermartingale.

%de2.5 #&#
\begin{definition}
A real-valued stochastic process $X=(X_t)_{0\leq t\leq1}$ is called an
\emph{optional strong supermartingale}, if:
\begin{longlist}[(3)]
\item[(1)] $X$ is optional;
\item[(2)] $X_\tau$ is integrable for every $[0,1]$-valued
stopping time $\tau$;
\item[(3)] for all stopping times $\sigma$ and $\tau$ with
$0\leq\sigma\leq\tau\leq1$, we have
\[
X_\sigma\geq E[X_\tau|\mathcal{F}_\sigma].
\]
\end{longlist}
\end{definition}

These processes have been introduced by Mertens \cite{M72} as a
generalization of the notion of a c\`adl\`ag (right continous with left
limits) supermartingale that one is usually working with. Indeed, by
the optional sampling theorem each c\`adl\`ag supermartingale is an
optional strong supermartingale, but not every optional strong
supermartingale has a c\`adl\`ag modification. For example, every
\textit{deterministic} decreasing function $(X_t)_{0 \leq t \leq1}$ is an
optional strong supermartingale, but there is little reason why it
should be c\`adl\`ag. However, by Theorem~4 in Appendix I in \cite
{DM82}, every optional strong supermartingale is indistinguishable from
a l\`adl\`ag (left and right limits) process, and so we can assume
without loss of generality that all optional strong supermartingales we
consider in this paper are l\`adl\`ag. Similarly to the Doob--Meyer
decomposition in the c\`adl\`ag case, every optional strong
supermartingale $X$ has a unique decomposition
%
%e2.3 #&#
\begin{equation}\label{eqMD}
X=M-A
\end{equation}
into a local martingale $M$ and a nondecreasing predictable
process $A$ starting at~$0$. This decomposition is due to Mertens \cite
{M72} (compare also Theorem~20 in Appendix~I in \cite{DM82}) and is
therefore called the \emph{Mertens decomposition}. Note that, under the
usual conditions of completeness and right continuity of the
filtration, we can and do choose a c\`adl\`ag modification of the local
martingale $M$ in \eqref{eqMD}. On the other hand, the nondecreasing
process $A$ is in particular l\`adl\`ag.

For l\`adl\`ag processes $X=(X_t)_{0 \leq t \leq1}$ we denote by
$X_{t+}:= \lim_{h \searrow0} X_{t+h}$ and $X_{t-}:=\lim_{h \searrow0}
X_{t-h}$ the right and left limits and by $\Delta_+X_t:=X_{t+} - X_t$
and $\Delta X_t:=X_t - X_{t-}$ the right and left jumps. We also use
the convention that $X_{0-}=0$ and $X_{1+} = X_1$.

After these preparations we have now everything in place to formulate
our main results. The proofs will be given in the Sections~\ref{sec3},
\ref{sec5}, \ref{sec6} and \ref{sec7}.

%th2.6 #&#
\begin{theorem}\label{c1}
Let $(M^n)^\infty_{n=1}$ be\vspace*{1pt} a sequence of nonnegative c\`adl\`ag
martingales $M^n=(M^n_t)_{0\leq t \leq1}$ starting at $M^n_0=1$. Then
there is a sequence $(\widetilde{M}^n)_{n=1}^\infty$ of convex combinations
\[
\widetilde{M}^n \in\operatorname{conv}\bigl(M^n,M^{n+1},
\ldots\bigr)
\]
and a nonnegative optional strong supermartingale $X=(X_t)_{0\leq
t\leq1}$ such that, for every $[0,1]$-valued stopping time $\tau$, we
have that
%
%e2.4 #&#
\begin{equation}
\label{M2} \widetilde{M}^n_\tau\stackrel{P}{\longrightarrow} X_\tau.
\end{equation}
\end{theorem}

Combining the above with a similar convergence result for predictable
finite variation processes by Campi and Schachermayer \cite{CS06}
allows us to extend our convergence result to optional strong
supermartingales by using the Mertens decomposition. Theorem~\ref{c1}
is thus only a special case of the following result.

%th2.7 #&#
\begin{theorem}\label{t1}
Let $(X^n)^\infty_{n=1}$ be a sequence of nonnegative optional strong
supermartingales $X^n=(X_t)_{0 \leq t \leq1}$ starting at $X^n_0=1$.\vspace*{1pt}
Then there is a sequence $(\widetilde{X}^n)_{n=1}^\infty$ of convex
combinations
\[
\widetilde{X}^n \in\operatorname{conv}\bigl(X^n,X^{n+1},
\ldots\bigr)
\]
and a nonnegative optional strong supermartingale $X=(X_t)_{0 \leq t
\leq1}$ such that, for every $[0,1]$-valued stopping time $\tau$, we
have convergence in probability, that is,
%
%e2.5 #&#
\begin{equation}
\label{C2} \widetilde{X}^n_\tau\stackrel{P} {
\longrightarrow} X_\tau.
\end{equation}
\end{theorem}

%Note
We thank Kostas Kardaras for indicating
that
convergence~\eqref{C2} is \emph{topological}. It corresponds
to the weak topology that is generated on the space of optional
processes by the topology of $L^0(P)$ and all evaluation mappings
$e_\tau(X) (\omega):= X_{\tau(\omega)}(\omega)$ that evaluate an optional
process $X=(X_t)_{0 \leq t \leq1}$ at a finite stopping time~$\tau$.
By the optional cross section theorem this topology is Hausdorff.

Given Theorem~\ref{c1} and Theorem~\ref{t1} above one can ask
conversely what the smallest class of stochastic processes is that is
closed under convergence in probability at all finite stopping times
and contains the set of bounded martingales.
Here the next result shows that this set is the set of optional strong
supermartingales.

%th2.8 #&#
\begin{theorem}\label{t2}
Let $X=(X_t)_{0 \leq t \leq1}$ be an optional strong supermartingale
and suppose that its stochastic base $(\Omega,\mathcal{F},\mathbb
{F},P)$ is sufficiently
rich to support a Brownian motion $W=(W_t)_{0\leq t\leq1}$. Then there
is a sequence of bounded c\`adl\`ag martingales $(M^n)^\infty_{n=1}$ such
that, for every $[0,1]$-valued stopping time $\tau$, we have
convergence in probability, that is,
%
%e2.6 #&#
\begin{equation}
\label{eqt2} M^n_\tau\stackrel{P}{\longrightarrow} X_\tau.
\end{equation}
\end{theorem}

We thank Perkowski and Ruf for pointing out to us that they have
independently obtained a similar result to Theorem~\ref{t2} for c\`adl\`ag supermartingales in Proposition~5.9 of \cite{PR} by taking several
limits successively. Moreover, we would like to thank Ruf for insisting
on a clarification of an earlier version of Theorem~\ref{t2} which led
us to a correction of the statement [convergence in probability in
\eqref{eqt2} as opposed to almost sure convergence] as well as to a
more detailed proof.

Let us now turn to the theme of stochastic integration. By Theorem~\ref{c1} the limit of a sequence $(M^n)^\infty_{n=1}$
of martingales in the
sense of \eqref{M2} will, in general, be no longer a semimartingale. In
order to come up with a similar convergence result for stochastic
integrals $\varphi\bolds{\cdot}M^n=\int
\varphi \,dM^n$, we therefore need to
restrict the choice of integrands $\varphi=(\varphi_t)_{0 \leq t \leq
1}$ to predictable finite variation processes. As we shall explain in
more detail in Section~\ref{sec7} below, this allows us to define
stochastic integrals $\varphi\bolds{\cdot}X=\int\varphi \,dX$ with respect
to optional strong supermartingales $X=(X_t)_{0 \leq t \leq1}$
pathwise, since $X$ is l\`adl\`ag. These integrals coincide with the
usual stochastic integrals, if $X=(X_t)_{0 \leq t \leq1}$ is a
semimartingale. For a general predictable, finite variation process
$\varphi$, the stochastic integral $\varphi\bolds{\cdot}X$ depends not only
on the values of the integrator $X$ but also explicitly on that of its
left limits $X_-$; see \eqref{defSI2} below. As a consequence, in
order to obtain a satisfactory convergence result for the integrals
$\varphi\bolds{\cdot}X^n$ to a limit $\varphi
\bolds{\cdot}X$, we have to take
special care of the left limits of the integrators. (The convergence of
stochastic integrals is crucially needed in applications in
mathematical finance, where the integrals correspond to the gains from
trading by using self-financing trading strategies.)\vspace*{-1pt} More precisely:
given the convergence $\widetilde{X}^n_\tau\stackrel{P}{\longrightarrow} X_\tau$ as in \eqref{C2}, at all $[0,1]$-valued stopping times
$\tau$
of a sequence $(\widetilde{X}^n)^\infty_{n=1}$ of optional strong
supermartingales do we have the convergence of the left limits
%
%e2.7 #&#
\begin{equation}
\label{eqcll}
\widetilde{X}^n_{\sigma-} \stackrel{P} {
\longrightarrow} X_{\sigma-}
\end{equation}
for all $[0,1]$-valued stopping times $\sigma$ as well?

For \emph{totally inaccessible} stopping times $\sigma$, we are able to
prove that \eqref{eqcll} is actually the case.

%pr2.9 #&#
\begin{proposition}\label{propti}
Let $(X^n)^{\infty}_{n=1}$ and $X$ be nonnegative optional strong
supermartingales $(X^n_t) _{0\leq t \leq1}$ and $(X_t) _{0\leq t \leq
1}$ such that
\[
X^n_q\stackrel{P} {\longrightarrow}X_q
\]
for every rational number $q\in[0,1]$. Then
\[
X^n_{\tau-}\stackrel{P} {\longrightarrow}X_{\tau-}
\]
for all $[0,1]$-valued \emph{totally inaccessible} stopping times
$\tau$.
\end{proposition}

At accessible stopping times $\sigma$, the convergence $\widetilde
{X}^n_\tau\stackrel{P}{\longrightarrow}X_{\tau}$ for all finite
stopping times $\tau$ does not necessarily imply\vspace*{1.5pt} convergence \eqref{eqcll} of the left limits $\widetilde{X}^n_{\sigma-}$. Moreover, even
if the left limits $\widetilde{X}^n_{\sigma-}$ converge to some
random variable
$Y$ in probability, it may happen that $Y \neq X_{\sigma-}$. In order
to take this phenomenon into account, we need to consider two processes
$X^{(0)}=(X^{(0)}_t)_{0 \leq t \leq1}$ and $X^{(1)}=(X^{(1)}_t)_{0
\leq t \leq1}$ that correspond\vspace*{1.5pt} to the limiting processes of the left
limits $\widetilde{X}^n_-$ and the processes $\widetilde{X}^n$ itself
or,\vspace*{1pt} alternatively, replace the time interval $I=[0,1]$ by the set
$\widetilde{I}
=[0,1]\times\{0,1\}$ with the lexicographic order. The set $\widetilde
{I}$ is
motivated by the \emph{Alexandroff double arrow space}. Equipping the
set $\widetilde{I}$ with the lexicographic order simply means that we
split every
point $t\in[0,1]$ into a left and a right point $(t,0)$ and $(t,1)$,
respectively, such that $(t,0) < (t,1)$, that $(t,0) \leq(s,0)$ if and
only if $t\leq s$ and that $(t,1) < (s,0)$ if and only if $t<s$. Then\vspace*{0.5pt}
we can merge both processes, $X^{(0)}=(X^{(0)}_t)_{0 \leq t \leq1}$
and $X^{(1)}=(X^{(1)}_t)_{0 \leq t \leq1}$, into one process,
%
%e2.8 #&#
\begin{equation}
\label{eqADAS} X_{\tilde{t}} = %
\cases{ X^{(0)}_t,
&\quad$\tilde{t} = (t,0)$,
\vspace*{2pt}\cr
X^{(1)}_t, &\quad$\tilde{t} =
(t,1)$, }
\end{equation}
for $\tilde{t}\in\widetilde{I}$, which is by \eqref{eqrelt3} below a
supermartingale indexed by $\tilde{t}\in\widetilde{I}$. As the limit
of the left
limits, the process $X^{(0)}=(X^{(0)}_t)_{0 \leq t \leq1}$ will be
predictable and it will turn out that it is even a predictable strong
supermartingale. We refer to the article of Chung and Glover \cite
{CG79} (see the second remark following the proof of Theorem~3 on page
243) as well as Definition~3 in Appendix~I of the book of Dellacherie
and Meyer \cite{DM82} for the subsequent concept.

%de2.10 #&#
\begin{definition}\label{defpred}
A real-valued stochastic process $X=(X_t)_{0 \leq t \leq1}$ is called
a \emph{predictable strong supermartingale} if:
\begin{longlist}[(3)]
\item[(1)] $X$ is predictable;
\item[(2)] $X_{\tau}$ is integrable for every $[0,1]$-valued
\emph{predictable} stopping time $\tau$;
\item[(3)] for all \emph{predictable} stopping times
$\sigma$
and $\tau$ with $0 \leq\sigma\leq\tau\leq1$, we have
\[
X_\sigma\geq E[X_\tau| \mathcal{F}_{\sigma-}].
\]
\end{longlist}
\end{definition}

After these preparations we are able to extend Theorem~\ref{t1} to hold
also for left limits.

%th2.11 #&#
\begin{theorem}\label{t3}
Let $(X^n)^\infty_{n=1}$ be a sequence\vspace*{1pt} of nonnegative optional strong
supermartingales starting at $X^n_0=1$. Then there is a sequence
$(\widetilde{X}
^n)^\infty_{n=1}$ of convex combinations $\widetilde{X}^n\in
\operatorname{conv}(X^n, X^{n+1},
\ldots)$, a nonnegative optional strong supermartingale
$X^{(1)}=(X^{(1)}_t)_{0 \leq t \leq1}$ and a nonnegative predictable
strong supermartingale $X^{(0)}=(X^{(0)}_t)_{0 \leq t \leq1}$ such that
%
%e2.9 #&#
%e2.10 #&#
\begin{eqnarray}\label{eqt31}
\widetilde{X}^n_{\tau}  &\stackrel{P}{\longrightarrow}&
X^{(1)}_\tau,
\\
\label{eqt32}
\widetilde{X}^n_{\tau-} &\stackrel{P} {\longrightarrow}&
X^{(0)}_\tau,
\end{eqnarray}
for all $[0,1]$-valued stopping times $\tau$, and we have that
%
%e2.11 #&#
\begin{equation}
X^{(1)}_{\tau-}\geq X^{(0)}_{\tau} \geq E
\bigl[X^{(1)}_{\tau} | \mathcal {F}_{\tau-}\bigr]
\label{eqrelt3}
\end{equation}
for all $[0,1]$-valued \emph{predictable} stopping times $\tau$.
\end{theorem}

With the above we can now formulate the following proposition. Note
that, since $\varphi\bolds{\cdot}\widetilde
{X}^n \in\operatorname{conv}(\varphi\bolds{\cdot}
X^n, \varphi\bolds{\cdot}X^{n+1}, \ldots)$,
part (2) is indeed an analogous
result to Theorem~\ref{t1} for stochastic integrals.

%pr2.12 #&#
\begin{proposition}\label{pSI}
Let $(X^n)^\infty_{n=1}$ be a sequence of nonnegative optional strong
supermartingales $X^n=(X^n_t)_{0\leq t\leq1}$ starting at $X^n_0 = 1$.
Then there exist convex combinations $\widetilde{X}^n \in
\operatorname{conv}(X^n,
X^{n+1}, \ldots)$ as well as an optional and a predictable strong
supermartingale $X^{(1)}$ and $X^{(0)}$ such that:
\begin{longlist}[(2)]
\item[(1)] $\widetilde{X}^n_\tau\stackrel{P}{\longrightarrow}
X^{(1)}_\tau$ and $\widetilde{X}^n_{\tau-} \stackrel
{P}{\longrightarrow
} X^{(0)}_\tau$ for all $[0,1]$-valued stopping times $\tau$;
\item[(2)] for all predictable processes $\varphi=(\varphi_t)_{0
\leq t \leq1}$ of finite variation, we have that
\[
\varphi \bolds{\cdot} {\widetilde{X}}^n_\tau
\stackrel{P}{\longrightarrow} \int^\tau_0
\varphi^c_u \,d X^{(1)}_u + \sum
_{0 < u \leq\tau} \Delta \varphi _u
\bigl(X^{(1)}_\tau- X^{(0)}_u\bigr) +
\sum_{0 \leq u < \tau} \Delta_+ \varphi _u
\bigl(X^{(1)}_\tau- X^{(1)}_u\bigr)
\]
for all $[0,1]$-valued stopping times $\tau$, where $\varphi^c$ denotes
the continuous part of $\varphi$, that is,
%
%e2.12 #&#
\begin{equation}\label{defcont}
\varphi^c_t:= \varphi_t - \sum
_{0 < u \leq t} \Delta\varphi_u - \sum
_{0 \leq u < t} \Delta_+ \varphi_u\qquad \mbox{for }t \in
[0,1].
\end{equation}
\end{longlist}
\end{proposition}

%s3 #&#
\section{Proof of Theorems \texorpdfstring{\protect\ref{c1}}{2.6} and \texorpdfstring{\protect\ref{t1}}{2.7}}\label{sec3}

The basic idea for the proof of Theorem~\ref{c1} is to consider the
Fatou limit $\overline{X}=(\overline{X}_t)_{0 \leq t \leq1}$ as
defined in \eqref{defFatou}.
Morally speaking $\overline{X}=(\overline{X}_t)_{0 \leq t \leq1}$
should also be the
limit of the sequence $(\overline{M})^\infty_{n=1}$ in the sense of
\eqref{M2}.
However, as we illustrated in Example~\ref{Ex0}, things may be more
delicate. While we do not need to have convergence in probability at
all finite stopping times in general, the next lemma shows that we
always have one-sided $P$-almost sure convergence.

%le3.1 #&#
\begin{lemma}\label{A1}
Let $\overline{X}$ and $(\overline{M}^n)^\infty_{n=1}$ be as in
Proposition~\ref{pFatou}. Then we have that
%
%e3.1 #&#
\begin{equation}
\label{chrisA1}
\bigl(\overline{M}^n_\tau-
\overline{X}_\tau\bigr)^- \mathop{\longrightarrow}^{P\mbox{-}\mathit{a.s.}} 0,\qquad
\mbox{as } n\to\infty,
\end{equation}
for all $[0,1]$-valued stopping times $\tau$, where $x^-=\max\{-x,0\}$.
\end{lemma}

\begin{pf}
Let $\sigma_k$ be the $k$th dyadic approximation of the stopping time
$\tau$, that is,
%
%e3.2 #&#
\begin{equation}\label{A1.1}
\sigma_k:=\inf\{t\in D_k|t>\tau\}\wedge1,
\end{equation}
where $D_k = \{j2^{-k} | j=0,\ldots,2^k\}$. As $\overline{M}^n$ is a
martingale,
we have $\overline{M}^n_{\tau} = E [\overline{M}^n_{\sigma_k} |
\mathcal{F}_{\tau}]$, for
every $n\in\mathbb{N}$, and therefore
\[
\mathop{\underline{\lim}}_{n\to\infty}\overline{M}^n_{\tau}=
\mathop{\underline{\lim}}_{n\to
\infty} E\bigl[\overline{M}
^n_{\sigma_k}\big|\mathcal{F}_\tau\bigr] \geq E \Bigl[
\mathop{\underline{\lim}}_{n\to\infty} \overline{M} ^n_{\sigma_k}
\Big| \mathcal{F}_{\tau}\Bigr] = E [Z_{\sigma_k} | \mathcal
{F}_{\tau}]
\]
for all $k$ by Fatou's lemma, where $Z_q$ is defined in Proposition~\ref{pFatou}, for every $q\in\mathbb{Q}\cap[0,1]$. Since $Z_{\sigma_k}
\to
\overline{X}_{\tau}$ $P$-a.s. and in $L^1(P)$ by backward
supermartingale convergence (see Theorem~V.30 and the proof of Theorem~IV.10 in \cite{DM82}, e.g.), we obtain that
\[
\mathop{\underline{\lim}}_{n\to\infty} \overline{M}^n_{\tau}
\geq\overline {X}_{\tau},
\]
which proves \eqref{chrisA1}.
\end{pf}
For any sequence $(\widehat{M}^n)^\infty_{n=1}$ of convex combinations
\[
\widehat{M}^n \in\operatorname{conv}\bigl(\overline{M}^n,
\overline {M}^{n+1}, \ldots\bigr),
\]
we can use the one-sided convergence \eqref{chrisA1} to show in the
next lemma that at any given stopping time $\tau$, we either have the
convergence of $\widehat{M}^n_\tau$ to $\overline{X}_\tau$ in
probability, or there
exists a sequence $(\widetilde{M}^n)^\infty_{n=1}$ of convex combinations
\[
\widetilde{M}^n \in\operatorname{conv}\bigl(\widehat{M}^n,
\widehat {M}^{n+1}, \ldots\bigr)
\]
and a nonnegative random variable $Y$ such that $\widetilde{M}^n_\tau
\stackrel{P}{\longrightarrow} Y$. In the latter case, $Y \geq
\overline{X}_\tau$
and $E[Y] > E[\overline{X}_\tau]$, as we shall now show.

%le3.2 #&#
\begin{lemma}\label{A2}
Let\vspace*{1.5pt} $\overline{X}$ and $(\overline{M}^n)^\infty_{n=1}$ be as in Proposition~\ref{pFatou}, let $\tau$ be a $[0,1]$-valued stopping time and
$(\widehat{M}
^n)^\infty_{n=1}$ a sequence of convex combinations $\widehat{M}^n
\in\operatorname{conv}
(\overline{M}^n, \overline{M}^{n+1}, \dots)$. Then we have\vspace*{-1pt} either
%
%e3.3 #&#
\begin{equation}
\label{A2.1}
\bigl(\widehat{M}^n_\tau-
\overline{X}_\tau\bigr)^+ \stackrel {P} {\longrightarrow} 0,\qquad \mbox{as }
n\to\infty,
\end{equation}
with $x^+=\max\{x,0\}$, or there exists a sequence $(\widetilde{M})^\infty_{n=1}$ of convex\vspace*{-1pt} combinations
\[
\widetilde{M}^n \in\operatorname{conv}\bigl(\widehat{M}^n,
\widehat {M}^{n+1},\ldots\bigr) \subseteq\operatorname{conv}\bigl(
\overline{M}^n, \overline{M} ^{n+1}, \ldots\bigr)
\]
and a nonnegative random variable $Y$ such\vspace*{-1pt} that
%
%e3.4 #&#
\begin{equation}
\label{A2.3}
\widetilde{M}^n_\tau\stackrel{P}{\longrightarrow} Y, \qquad\mbox{as } n\to\infty,
\end{equation}
and\vspace*{-3pt}
%
%e3.5 #&#
\begin{equation}
\label{A2.4}
E[Y_\tau] > E[\overline{X}_{\tau}].
\end{equation}
\end{lemma}

\begin{pf}
If \eqref{A2.1} does not hold, there exists $\alpha>0$ and a
subsequence $(\widehat{M}^n)$, still denoted by $(\widehat
{M}^n)^\infty
_{n=1}$ again indexed by $n$,\vspace*{-1pt} such that
%
%e3.6 #&#
\begin{equation}
\label{A2.56}
P\bigl(\widehat{M}^n_\tau-
\overline{X}_\tau>\alpha\bigr)\geq\alpha
\end{equation}
for all $n$. %, as $\mathop{\underline{\lim}}_{n\to\infty} \hM^n_\tau\geq\bX_\tau$
%by Lemma~\ref{A1}.
Since $E[\widehat{M}^n_\tau]=1$, there exists by Koml\'os's lemma a
sequence $(\widetilde{M}^n)^\infty_{n=1}$ of convex combinations
$\widetilde{M}^n \in\operatorname{conv}(\widehat{M}^n, \widehat{M}^{n+1}, \ldots)$ and a
nonnegative random variable $Y$ such that \eqref{A2.3} holds.
To see \eqref{A2.4}, we observe that, for\vspace*{-1pt} each $\varepsilon>0$,
\[
\mathbh{1}_{\{\widehat{M}^n_\tau\geq\overline{X}_\tau
-\varepsilon\}} \stackrel{P}{\longrightarrow} 1,\qquad \mbox{as }  n\to
\infty,
\]
by \eqref{chrisA1}. From the\vspace*{-1pt} inequality
\[
\widehat{M}^n_\tau\mathbh{1}_{A_n} \geq
\overline{X}_\tau\mathbh{1}_{A_n} + \alpha
\mathbh{1}_{A_n},
\]
where $A_n:=\{\widehat{M}^n_\tau\geq\overline{X}_\tau+\alpha\}$,
we\vspace*{-1pt} obtain
\[
\widehat{M}^n_\tau\mathbh{1}_{\{\widehat{M}^n_\tau\geq\overline
{X}_\tau-\varepsilon\}} \geq
\overline{X}_\tau\mathbh{1}_{\{\widehat{M}^n_\tau\geq\overline
{X}_\tau-\varepsilon\}} +\alpha\mathbh{1}_{A_n}.
\]
Now taking the convex combinations leading to $\widetilde{M}^n$ and\vspace*{-1pt} then
\[
\widetilde{Y}^n \in\operatorname{conv}
(\alpha\mathbh{1}_{A_n}, \alpha\mathbh{1}_{A_{n+1}},\ldots)
\]
such that $\widetilde{Y}^n \stackrel{P}{\longrightarrow} \widetilde{Y}$, as $n\to\infty$, we\vspace*{-1pt} derive
%
%e3.7 #&#
\begin{equation}
\label{A2.57}
Y \geq\overline{X}_\tau+\widetilde{Y}-\varepsilon
\end{equation}
by passing to limits.
Since $|\widetilde{Y}^n |\le1$ and $E[\widetilde{Y}^n]\geq\alpha^2$,
we deduce from\break Lebesgue's theorem that $\widetilde{Y}^n \stackrel
{L^1(P)}{\longrightarrow} \widetilde{Y}$, as $n\to\infty$, and
$E[\widetilde{Y}]\geq\alpha^2$. Therefore \eqref{A2.57} implies that
\[
E[Y] \geq E[\overline{X}_\tau] +E[\widetilde{Y}]-\varepsilon\geq E[
\overline {X}_\tau] +\alpha^2-\varepsilon
\]
for each\vadjust{\goodbreak} $\varepsilon>0$ and hence \eqref{A2.4} by sending
$\varepsilon\to0$.
\end{pf}
By the previous lemma we either already have the convergence of
$\widehat{M}^n_\tau$ to $\overline{X}_\tau$ in probability at a
given stopping\vspace*{2pt}
time $\tau$, or we can use Koml\'os's lemma once again to find convex
combinations\vspace*{-1pt} $\widetilde{M}^n \in\operatorname{conv}(\widehat{M}^n,
\widehat
{M}^{n+1}, \ldots)$ and a random variable $Y$ such that $\widetilde
{M}^n_\tau\stackrel{P}{\longrightarrow} Y$. The next lemma shows that
we can exhaust this latter phenomenon by a countable number\vspace*{1pt} of stopping
times $(\tau_m)^\infty_{m=1}$ and that we can use the random variables
$Y_m:=P-\lim_{n\to\infty} \widetilde{M}^n_{\tau_m}$ to redefine
the c\`adl\`ag supermartingale $\overline{X}$ at the stopping\vspace*{1pt} times $\tau_m$
to obtain
a limiting process $\widetilde{X}=(\widetilde{X}_t)_{0\leq t\leq1}$. The
limiting process $\widetilde{X}$ will be an optional strong
supermartingale, and
we can relate the loss of mass $Y_m- \overline{X}_{\tau_m}$ to the
right jumps\vspace*{1.5pt}
$\Delta_+ \widetilde{A}_{\tau_m}$ of the predictable part of the
Mertens decomposition $\widetilde{X} = \widetilde{M} - \widetilde{A}$.

%le3.3 #&#
\begin{lemma}\label{lE6}
In the setting of Proposition~\ref{pFatou}, let $(\tau
_m)_{m=1}^\infty
$ be a sequence of $[0,1]\cup\{\infty\}$-valued stopping times with
disjoint\vspace*{1pt} graphs, that is, $\llbracket\tau_m\rrbracket\cap\llbracket
\tau_k \rrbracket= \varnothing$ for $m\ne k$. Then there exists a
sequence $(\widetilde{M}^n)_{n=1}^\infty$ of convex combinations
$\widetilde{M}^n\in\operatorname{conv}
(\overline{M}^n,\overline{M}^{n+1},\ldots)$ such that, for each
$m\in\mathbb{N}$, the\vspace*{1pt} sequence
$(\widetilde{M}^n_{\tau_m})^\infty_{n=1}$ converges $P$-a.s. to a
random variable
$Y_m$ on $\{\tau_m < \infty\}$. The process $\widetilde{X}=(\widetilde{X}_t)_{0\le t\le1}$ given by
%
%e3.8 #&#
\begin{equation}
\label{eqlsup}
\widetilde{X}_t(\omega) = \cases{
Y_{m}(\omega), & \quad $t=\tau_m(\omega)<\infty$ and
$m\in\mathbb {N}$,\vspace*{3pt}\cr
\overline{X}_t(\omega), &\quad$\mbox{elsewhere}$}
\end{equation}
is an optional strong supermartingale with the following properties:
\begin{longlist}[(2)]
\item[(1)] $\widetilde{X}_+=\overline{X}$,\vspace*{1.5pt} where $\widetilde{X}_+$
denotes the process of the right limits of $\widetilde{X}$;
\item[(2)] denoting by $\widetilde{X} =\widetilde{M}-\widetilde{A}$, the
Mertens decomposition of $\widetilde{X}$, we have
%
%e3.9 #&#
\begin{equation}
\label{E6}
\widetilde{X}_{\tau_m} -\overline{X}_{\tau_m} = -
\Delta_+ \widetilde {X}_{\tau_m} = \Delta_+\widetilde{A}_{\tau_m} :=
\widetilde {A}_{{\tau
_m}_+} -\widetilde{A}_{\tau_m}
\end{equation}
for each $m\in\mathbb{N}$.
\end{longlist}
\end{lemma}

\begin{pf}
Combining Koml\'os's lemma with a diagonalization procedure we obtain
nonnegative random variables $Y_m$ and convex combinations
$\widetilde{M}^n \in\operatorname{conv}(\overline{M}^n, \overline
{M}^{n+1},\ldots)$ such that
\[
\widetilde{M}^n_{\tau_m} \mathop{\longrightarrow}^{P\mbox{-}\mathrm{a.s.}}
Y_m,
\]
for all $m\in\mathbb{N}$, and we can define the process $\widetilde
{X}$ via \eqref{eqlsup}. This process $\widetilde{X}$ is clearly optional.

To show that $\widetilde{X}$ is indeed an optional strong
supermartingale, we need to verify that
%
%e3.10 #&#
\begin{equation}
\label{lE61}
\widetilde{X}_{\varrho_1} \geq E[\widetilde{X}_{\varrho_2}
|\mathcal{F}_{\varrho_1}]
\end{equation}
for every pair of $[0,1]$-valued stopping times $\varrho_1$ and
$\varrho_2$
such that $\varrho_1\leq\varrho_2$. For this, we observe that it is
sufficient
to consider \eqref{lE61} on the set $\{\varrho_1 <\varrho_2\}$.
For $i=1,2$ denote by $(\varrho_{i,k})_{k=1}^\infty$ the $k$th dyadic
approximation of $\varrho_i$ as in \eqref{A1.1} above. Then we have
%
%e3.11 #&#
%e3.12 #&#
%e3.13 #&#
%e3.14 #&#
%e3.15 #&#
%e3.16 #&#
%e3.17 #&#
\begin{eqnarray}
&& E[\widetilde{X}_{\varrho_2}|\mathcal{F}_{\varrho_1}]\nonumber\\
&&\qquad= E \Biggl[\lim_{n\to\infty} \sum_{m=1}^\infty
\widetilde{M}^n_{\tau_m}\mathbh{1}_{\{\tau_m=\varrho
_2\}} +\lim_{k\to\infty} \Bigl(\lim_{n\to\infty} \overline
{M}^n_{\varrho_{2,k}} \Bigr)\mathbh{1}_{\{\tau_m\ne\varrho_2, \forall m\}} \Big|\mathcal
{F}_{\varrho_1} \Biggr]
\nonumber
\\
&&\qquad=E \Biggl[\lim_{n\to\infty} \sum_{m=1}^\infty
\widetilde {M}^n_{\tau
_m}\mathbh{1}_{\{\tau_m=\varrho_2\}} +\lim
_{k\to\infty} \Bigl(\lim_{n\to\infty} \widetilde{M}^n_{\varrho_{2,k}}
\Bigr)\mathbh{1}_{\{\tau_m\ne
\varrho_2, \forall m\}} \Big|\mathcal{F}_{\varrho_1} \Biggr]
\nonumber
\\
&&\qquad\leq   E \Biggl[\lim_{n\to\infty} \sum_{m=1}^\infty
\widetilde{M}^n_{\tau
_m}\mathbh{1}_{\{\tau_m=\varrho_2\}}
\nonumber
\\[-8pt]
\label{line3}\\[-8pt]
\nonumber&&\hspace*{46pt}{}+\lim
_{k\to\infty} \Bigl(\lim_{n\to\infty} E\bigl[
\widetilde{M}^n_{\varrho_{2,k}}|\mathcal{F}_{\varrho_{2}}\bigr]
\Bigr)\mathbh{1}_{\{\tau_m\ne\varrho_2,
\forall m\}} \Big|\mathcal{F}_{\varrho_1} \Biggr]
\\
\label{line4}
&&\qquad=  E \Bigl[\lim_{n\to\infty}\widetilde{M}^n_{\varrho_2}
\big|\mathcal{F}_{\varrho_1} \Bigr]
\\
\label{line5}
&&\qquad\leq  E \Bigl[\lim_{k\to\infty}\lim_{n\to\infty}E\bigl[
\widetilde {M}^n_{\varrho_2}|\mathcal{F}_{\varrho
_{1,k}}\bigr] \big|
\mathcal{F}_{\varrho_1} \Bigr]
\\
\label{line6}
&&\qquad=E \Bigl[\lim_{k\to\infty}\lim_{n\to\infty}\widetilde
{M}^n_{\varrho_{1,k}} \big|\mathcal{F} _{\varrho_1}
\Bigr]
\\
&&\qquad= E \Biggl[\lim_{k\to\infty}\lim_{n\to\infty} \sum
_{m=1}^\infty\widetilde {M}^n_{\varrho_{1,k}}
\mathbh{1}_{\{\tau_m=\varrho_1\}} +\lim_{k\to\infty}\lim_{n\to
\infty}
\widetilde{M}^n_{\varrho_{1,k}}\mathbh{1}_{\{\tau_m\ne
\varrho_1, \forall m\}} \Big|
\mathcal{F}_{\varrho_1} \Biggr]
\nonumber
\\
&&\qquad\le  \lim_{k\to\infty}\lim_{n\to\infty} \sum
_{m=1}^\infty E\bigl[\widetilde {M}^n_{\varrho_{1,k}}|
\mathcal{F}_{\varrho_1}\bigr]\mathbh{1}_{\{\tau
_m=\varrho_1\}}
\nonumber
\\[-8pt]
\label{line8}\\[-8pt]
\nonumber
&&\quad\qquad{}+E \Bigl[\lim
_{k\to\infty}\lim_{n\to\infty} \overline{M}^n_{\varrho_{1,k}}
\big|\mathcal{F}_{\varrho_1} \Bigr]\mathbh{1}_{\{\tau_m\ne\varrho_1, \forall m\}}
\\
\label{line9}
&&\qquad=\lim_{n\to\infty}\sum_{m=1}^\infty
\widetilde{M}^n_{\tau
_m}\mathbh{1}_{\{\tau_m=\varrho_1\}} +E \Bigl[
\lim_{k\to\infty} Z_{\varrho_{1,k}} \big|\mathcal{F}_{\varrho_1}
\Bigr]\mathbh{1}_{\{\tau_m\ne
\varrho_1, \forall m\}
}
\\
\label{line10}
&&\qquad=\sum_{m=1}^\infty\widetilde{X}_{\tau_m}
\mathbh{1}_{\{\tau
_m=\varrho_1\}
} +\overline{X}_{\varrho_1}\mathbh{1}_{\{\tau_m\ne\varrho_1,
\forall m\}}=
\widetilde{X}_{\varrho
_1}
\end{eqnarray}
by\vspace*{1.5pt} using Fatou's lemma in \eqref{line3}, \eqref{line5} and \eqref
{line8}, the martingale property of the $\widetilde{M}^n$ and the
convergence in probability of the $M^n$ in \eqref{line4}, \eqref{line6}
and \eqref{line9} and exploiting the backward supermartingale
convergence of $(Z_{\varrho_{1,k}})_{k=1}^\infty$ in \eqref{line10}.
\begin{longlist}[(2)]
\item[(1)]  We argue by contradiction and assume that $G:=\{\widetilde{X}_+
\neq \overline{X}\}$ has $P(\pi(G))>0$, where $\pi\dvtx \Omega\times[0,1]
\to\Omega$ is
given by $\pi ((\omega,t) )=\omega$. As the set $G$ is
optional, there
exists by the optional cross-section theorem (Theorem IV.84 in \cite{DM82})
a~$[0,1]\cup\{\infty\}$-valued stopping\vspace*{1pt} time $\sigma$ such
that $\llbracket\sigma_{\{\sigma<\infty\}} \rrbracket\subseteq G$ and
\mbox{$P(\sigma< \infty) > 0$}, which is equivalent to the assumption that
the set $F:=\{\widetilde{X}_{\sigma_+}\ne\overline{X}_\sigma\}$
has strictly
positive measure $P(F)>0$. Without loss of generality we can assume
that there exists \mbox{$\delta>0$} such that $F\subseteq\{\sigma+\delta
<1\}$.
Let $(h_i)_{i=1}^\infty$ be a sequence of real numbers decreasing to
$0$ that are no atoms of the laws $\tau_m-\sigma$ for all $m\in
\mathbb{N}$.
Then defining $\sigma_i:=(\sigma+h_i)_F\wedge1$ for each $i\in
\mathbb{N}$
gives a sequence\vspace*{1pt} of stopping times such that $\widetilde{X}_{\sigma_i}
=\overline{X}_{\sigma_i}$ for each $i$ and
$\sigma_i \searrow\sigma$ on $F$. But this implies that
%
%e3.18 #&#
\begin{equation}
\widetilde{X}_{\sigma+} =\lim_{i\to\infty} \widetilde
{X}_{\sigma_i} =\lim_{i\to\infty} \overline{X}_{\sigma_i}=
\overline {X}_\sigma\qquad \mbox{on }  F,
\end{equation}
which contradicts $P(F)>0$ and hence also $P(\pi(G))> 0$.

\item[(2)]  By property\vspace*{2pt} (1), modifying $\overline{X}$ at countably many
stopping times $(\tau_m)^\infty_{m=1}$ to obtain $\widetilde{X}$ leaves
right limits of the l\`adl\`ag optional strong supermartingale
$\widetilde{X}$ invariant so that these remain
%
%e3.19 #&#
\begin{equation}
\widetilde{X}_{\tau_m +} = \overline{X}_{\tau_m^+}=\overline
{X}_{\tau
_m}\qquad \mbox{on } \{\tau_m <1\} \mbox{ for each }  m.
\end{equation}
Since $\widetilde{M}$ is c\`adl\`ag, this implies that
%
%e3.20 #&#
\begin{equation}
\widetilde{X}_{\tau_m} -\overline{X}_{\tau_m} = -\Delta_+
\widetilde {X}_{\tau_m} =\Delta_+ \widetilde{A}_{\tau_m}
\end{equation}
for each $m$, thus proving property (2).\quad\qed
\end{longlist}
\noqed\end{pf}

Continuing with the proof of Theorem~\ref{c1}, the idea is to define
the limiting supermartingale $X$ by \eqref{eqlsup} and to use Lemma~\ref{lE6} to enforce the convergence at a well-chosen \emph{countable
number} of stopping times $(\tau_m)_{m=1}^\infty$ to obtain the
convergence in \eqref{C2} for \emph{all} stopping times. It is rather
intuitive that one has to take special care of the jumps of the
limiting process $X$. As these can be exhausted by a sequence $(\tau
_k)^{\infty}_{k=1}$ of stopping times, the previous lemma can take care
of this issue. However, the subsequent example shows that there also
may be a problem with the convergence in \eqref{M2} at a stopping time
$\tau$ at which $\overline{X}$ is \textit{continuous}.

%ex3.4 #&#
\begin{example}
Let $\sigma\dvtx \Omega\longrightarrow[0,1]$ be a \emph{totally
inaccessible} stopping time, $(A_t)_{0 < t \leq1}$ its\vspace*{1pt} compensator so
that $(\mathbh{1}_{\llbracket\sigma, 1 \rrbracket} (t) - A_t)_{0
\leq
t \leq1}$ is a martingale. Let $(Y_n)^\infty_{n=1}$ be a sequence of
random variables independent of $\sigma$ such that $Y_n$ takes values
in $ \{0, n\}$ and $P[Y_n = n] = \frac{1}{n}$. Define the \emph{continuous} supermartingale
\[
X^1_t = 1 - A_t,\qquad 0 \leq t \leq 1,
\]
and the optional strong supermartingale
\[
X^2_t = 1 - A_t + \mathbh{1}_{\llbracket\sigma\rrbracket}
(t),\qquad 0 \leq t \leq1.
\]
Define the sequences $(M^{1,n})^\infty_{n=1}$ and $(M^{2,n})^\infty
_{n=1}$ of martingales by
\begin{eqnarray*}
M^{1,n}_t &=& 1 - A_t +Y_n
\mathbh{1}_{\llbracket\sigma,1\rrbracket} (t),
\\
M^{2,n}_t &=& 1 - A_t + \mathbh{1}_{\llbracket\sigma,1\rrbracket}
(t) + (Y_n-1)\mathbh{1}_{\llbracket\sigma+{1}/{n},1\rrbracket} (t)
\end{eqnarray*}
for $t \in[0,1]$ and $n \in\mathbb{N}$. Then we have that
%
%e3.21 #&#
\begin{eqnarray}
M^{1,n}_\tau & \stackrel{P}\longrightarrow &
X^1_\tau,
\nonumber
\\[-8pt]
\label{eqF11}\\[-8pt]
\nonumber
M^{2,n}_\tau & \stackrel{P}\longrightarrow &
X^2_\tau
\end{eqnarray}
for all $[0,1]$-valued stopping times $\tau$.
The\vspace*{1pt} left and right limits of $X^1$ and $X^2$ coincide, that is,
$X^1_{-}=X^2_-$ and $X^1_{+}=X^2_{+}$, but $X^1\ne X^2$. As $X^1 =
X^1_- = X^1_{+} = X^2_{+}$ coincides with the Fatou limits $\overline
{X}^1$ (and $\overline{X}^2$, resp.) of $(M^{1,n})^\infty_{n=1}$ [and
$(M^{2,n})^\infty_{n=1}$, resp.], this example illustrates that we cannot
deduce from the Fatou limits $\overline{X}^1$ and $\overline{X}^2$,
where it is necessary to correct the convergence by using Lemma~\ref{lE6}. Computing the Mertens decompositions $X^1=M^1-A^1$ and
$X^2=M^2-A^2$, we obtain
\begin{eqnarray*}
M^1&= &1,
\\
A^1&=& \sigma\wedge t,
\\
M^2&=& 1-\sigma\wedge t+\mathbh{1}_{\llbracket\sigma,1\rrbracket
},
\\
A^2&=& \mathbh{1}_{\rrbracket\sigma,1\rrbracket}.
\end{eqnarray*}
This shows that using $X^2$ instead of $\overline{X}^2=X^1$ changes the
compensator of $M^2$ not only after the correction in the sense of
Lemma~\ref{lE6} on $\rrbracket\sigma,1\rrbracket$ but on all of $[0,1]$.
\end{example}

As the previous example shows, it might be difficult to identify the
stopping times $(\tau_m)^\infty_{m=1}$, where one needs to enforce the
convergence in probability by using Lemma~\ref{lE6}. Therefore we
combine the previous lemmas with an exhaustion argument to prove
Theorem~\ref{c1}.
\begin{pf*}{Proof of Theorem~\ref{c1}}
Let\vspace*{1pt} $\mathbb{T}$ be the collection of all families $\mathcal{T}=(\tau
_m)^{N(\mathcal{T}
)}_{m=1}$ of finitely many
$[0,1]\cup\{\infty\}$-valued stopping times $\tau_m$ with disjoint
graphs. For each $\mathcal{T}\in\mathbb{T}$, we consider an
optional strong supermartingale $X^\mathcal{T}$ that is obtained by taking
convex combinations $\widetilde{X}^{n,\mathcal{T}}\in\operatorname
{conv}(\overline{M}^n,\overline{M}^{n+1},\ldots
)$ such that $\widetilde{X}^{n,\mathcal{T}}_{\tau_m} \stackrel{P}{\longrightarrow} Y^\mathcal{T}_m$ on $\{\tau_m < \infty\}$ for each $m=1,\ldots,
N(\mathcal{T})$ and
then setting
%
%e3.22 #&#
\begin{equation}
X^\mathcal{T}_t(\omega) = %
\cases{
Y^\mathcal{T}_m(\omega), & \quad$t=\tau_m(\omega)<
\infty$ and $m=1,\ldots, N(\mathcal{T})$,\vspace*{3pt}
\cr
\overline{X}_t(
\omega), &\quad$\mbox{else}$,}
\end{equation}
as explained in Lemma~\ref{lE6}. Then each $X^\mathcal{T}$ has a
Mertens decomposition
%
%e3.23 #&#
\begin{equation}
X^\mathcal{T}= M^\mathcal{T}- A^\mathcal{T},
\end{equation}
and we have by part (2) of Lemma~\ref{lE6} that
\[
E \Biggl[ \sum^{N(\mathcal{T})}_{m=1}
\bigl(X^\mathcal{T}_{\tau
_{m}\wedge1} - \overline {X}_{\tau_{m}\wedge1}\bigr)
\Biggr]= E \Biggl[\sum^{N(\mathcal
{T})}_{m=1} \Delta _+
A^\mathcal{T}_{\tau_{m}\wedge1} \Biggr] \le1.
\]

Therefore
%
%e3.24 #&#
\begin{equation}
\widehat{\vartheta}:=\sup_{\mathcal{T}\in\mathbb{T}} E \Biggl[ \sum
^{N(\mathcal{T})}_{m=1} \bigl(X^\mathcal{T}_{\tau_{m}\wedge1}
-\overline{X}_{\tau_{m}\wedge
1}\bigr) \Biggr] \le1,
\end{equation}
and there exists a maximizing sequence $(\mathcal{T}_k)_{k=1}^\infty$
such that
%
%e3.25 #&#
\begin{equation}
\hspace*{10pt}E \Biggl[ \sum^{N(\mathcal{T}_k)}_{m=1}
\bigl(X^{\mathcal
{T}_k}_{\tau_{m}\wedge1} -\overline{X}_{\tau_{m}\wedge1}\bigr) \Biggr]
\nearrow\sup_{\mathcal{T}\in
\mathbb{T}} E \Biggl[\sum^{N(\mathcal{T})}_{m=1}
\bigl(X^{\mathcal
{T}}_{\tau_{m}\wedge1} -\overline{X}_{\tau_{m}\wedge1}\bigr)
\Biggr] =\widehat{\vartheta}.
\end{equation}
It is easy to see that we can assume that $(\mathcal{T}_k)^\infty
_{k=1}$ can be
chosen to be increasing, that is, $\mathcal{T}_k\subseteq\mathcal
{T}_{k+1}$ for each $k$.
This means that $\mathcal{T}_{k+1}$ just adds some \mbox{stopping} times to
those which
appear in $\mathcal{T}_k$. Then $\widetilde{\mathcal{T}}:=\bigcup
^\infty_{k=1} \mathcal{T}_k$ is a
countable collection of stopping times $(\tau_m)^\infty_{m=1}$ with
disjoint graphs, and by Lemma~\ref{lE6} there exists an optional strong
supermartingale $X^{\widetilde{\mathcal{T}}}$ and convex combinations
$X^{n,\widetilde{\mathcal{T}}}\in\operatorname{conv}(\overline
{M}^n,\overline{M}^{n+1},\ldots
)$ such that $X^{n,\widetilde{\mathcal{T}}}_{\widetilde{\tau}_m}
\stackrel
{P}{\longrightarrow} Y^{\widetilde{\mathcal{T}}}_m$ for all $m$ and
%
%e3.26 #&#
\begin{equation}
X^{\widetilde{\mathcal{T}}}_t (\omega):= %
\cases{
Y^{\widetilde{\mathcal{T}}}_m(\omega), & \quad$t=\tau_m(\omega )<
\infty$,\vspace*{3pt}
\cr
\overline{X}_t(\omega), &\quad$
\mbox{else}$.}
\end{equation}
As we can\vspace*{1pt} suppose
without loss of generality
that $X^{n,\mathcal{T}_{k+1}}\in\operatorname{conv}(X^{n,\mathcal
{T}_k},\break  X^{n+1,\mathcal{T}_k},\ldots)$ and
$X^{n,\widetilde{\mathcal{T}}}\in\operatorname{conv}(X^{n,\mathcal
{T}_k}, X^{n+1,\mathcal{T}_{n+1}},\ldots)$,
we have that $Y^{\mathcal{T}_k}_m= Y^{\mathcal{T}_{k+1}}_m =
Y_m^{\widetilde{\mathcal{T}}}$ on $\{
\tau_m <1\}$ for all $k\geq m$. Let $X^{\widetilde{\mathcal{T}}} =
M^{\widetilde
{\mathcal{T}}} -A^{\widetilde{\mathcal{T}}}$ be the Mertens
decomposition of
$X^{\widetilde{\mathcal{T}}}$. Then
%
%e3.27 #&#
\begin{equation}
\Delta_+ A^{\widetilde{\mathcal{T}}}_{\tau_m} = X^{\widetilde
{\mathcal{T}}}_{\tau_m} -
\overline{X}_{\tau_m} = X^{\mathcal{T}_k}_{\tau_m} -\overline
{X}_{\tau_m} =\Delta_+ A^{\mathcal{T}_k}_{\tau_m}
\end{equation}
on $\{\tau_m <1\}$ for $m\le N(\mathcal{T}_k)$, since, as we
explained in the
proof of Lemma~\ref{lE6}, modifying $\overline{X}$ at countably many
stopping times does not change the right limits, and these remain
%
%e3.28 #&#
\begin{equation}
X^{\widetilde{\mathcal{T}}}_{\tau_m +} =\overline{X}_{\tau_m} =
X^{\mathcal{T}_k}_{\tau
_m+} \qquad\mbox{on $\{\tau_m < 1\}$ for $m\le
N(\mathcal{T}_k)$.}
\end{equation}
This implies that
%
%e3.29 #&#
\begin{equation}
\hspace*{8pt}\sum^{N(\mathcal{T}_k)}_{m=1} \bigl(X^{\mathcal{T}_k}_{\tau
_{m}\wedge1}
- \overline {X}_{\tau_{m}\wedge1}\bigr) =\sum^{N(\mathcal{T}_k)}_{m=1}
\bigl(X^{\widetilde{\mathcal{T}
}}_{\tau_{m}\wedge1} -\overline{X}_{\tau_{m}\wedge1}\bigr) =\sum
^{N(\mathcal{T}_k)}_{m=1} \Delta_+ A^{\widetilde{\mathcal
{T}}}_{\tau_{m}\wedge1}
\end{equation}
and therefore
%
%e3.30 #&#
\begin{equation}
E \Biggl[ \sum^\infty_{m=1} \Delta_+
A^{\widetilde{\mathcal
{T}}}_{\tau
_{m}\wedge1} \Biggr] = E \Biggl[ \sum
^\infty_{m=1} \bigl(X^{\widetilde{\mathcal{T}
}}_{\tau_{m}\wedge1} -
\overline{X}_{\tau_{m}\wedge1}\bigr) \Biggr] = \widehat{\vartheta}
\end{equation}
by the monotone convergence theorem.

Now suppose that there exists a $[0,1]$-valued stopping time $\tau$
such that $X^{n,{\widetilde{\mathcal{T}}}}_\tau$ does not converge in
probability to $X^{\widetilde{\mathcal{T}}}_\tau$. By Lemma~\ref
{A2} we can then
pass once more to convex
combinations $\widetilde{M}^n \in\operatorname
{conv}(X^{n,{\widetilde{\mathcal{T}}}},
X^{n+1,{\widetilde{\mathcal{T}}}},\ldots)$ such that there exists a random
variable $Y$ such that
$\widetilde{M}^n_\tau\stackrel{P}{\longrightarrow} Y$, $\widetilde
{M}^n_{\tau_m} \stackrel{P}{\longrightarrow} Y^{\widetilde{\mathcal
{T}}}_m$ and
an optional strong supermartingale $\widetilde{X}$ such that
%
%e3.31 #&#
\begin{equation}
\widetilde{X}_t(\omega) = %
\cases{ Y(\omega), &\quad $t=
\tau(\omega) \leq1$, \vspace*{3pt}
\cr
X^{\widetilde{\mathcal{T}}}_t(\omega), &
\quad$\mbox{else}$.}
\end{equation}
However, since $E[\widetilde{X}_\tau-\overline{X}_\tau]>0$ by Lemma~\ref{A2}, setting ${\widetilde{\mathcal{T}}}_k:=\mathcal{T}_k\cup\{
\mathcal{T}\}$ gives a
sequence in $\mathbb{T}$ such that
\begin{eqnarray*}
&& \lim_{k\to\infty} E \Biggl[\sum^{N({\widetilde
{\mathcal{T}}}_k)}_{m=1}
\bigl(X^{\widetilde{\mathcal{T}}_k}_{\tau_{m}\wedge1} -\overline {X}^{{\widetilde{\mathcal{T}
}}_k}_{\tau_{m}\wedge1}
\bigr) \Biggr]\\
&&\qquad =\lim_{k\to\infty} E \Biggl[\sum
^{N(\mathcal{T}_k)}_{m=1} \bigl(X^{\mathcal{T}_k}_{\tau_{m}\wedge
1} -
\overline {X}_{\tau_{m}\wedge1}\bigr) \Biggr] + E[\widetilde{X}_\tau-
\overline {X}_\tau ]
\\
&&\qquad= \widehat{\vartheta} + E[\widetilde{X}_\tau-\overline{X}_\tau]
> \widehat{\vartheta},
\end{eqnarray*}
and therefore a contradiction to the definition of $\widehat{\vartheta
}$ as supremum. Here we can take the convex combinations $\widetilde
{M}^n \in\operatorname{conv}(X^{n,{\widetilde{\mathcal{T}}}},
X^{n+1,{\widetilde{\mathcal{T}}}},\ldots
)$ for all
${\widetilde{\mathcal{T}}}_k$.
\end{pf*}
Combining Theorem~\ref{c1} with a similar convergence result for
predictable finite variation processes by Campi and Schachermayer \cite
{CS06}, we now deduce Theorem~\ref{t1} from Theorem~\ref{c1}.
\begin{pf*}{Proof of Theorem~\ref{t1}}
We consider the extension of Theorem~\ref{c1} to local martingales
first. For this, let $(X^n)^\infty_{n=1}$ be a sequence of nonnegative
local martingales $X^n=(X^n_t) _{0\leq t \leq1}$ and $(\sigma
^n_m)^{\infty}_{m=1}$ a localizing sequence of $[0,1]$-valued stopping
times for each $X^n$. Then, for each $n\in\mathbb{N}$, there exists
$m(n)\in\mathbb{N}$
such that $P(\sigma^n_m <1)<2^{-(n+1)}$ for all $m\geq m(n)$. Define
the martingales
%
%e3.32 #&#
\begin{equation}
M^n:=\bigl(X^n\bigr)^{\sigma^n_{m(n)}}
\end{equation}
that satisfy $M^k=X^k$ for all $k\geq n$ on $F_n:= \bigcap_{k\geq n} \{\sigma^k_{m(k)} =1\}$ with $P(F_n)>1-2^{-n}$.
By Theorem~\ref{c1} there exist a sequence of convex combinations
$\widetilde{M}
^n \in\operatorname{conv}(M^n, M^{n+1}, \dots)$ and an optional strong
supermartingale $X$ such that
\[
\widetilde{M}^k_{\tau} \stackrel{P} {\longrightarrow}
X_{\tau} \qquad\mbox{on $F_n$}
\]
for all $[0,1]$-valued stopping\vspace*{1.5pt} times $\tau$. Therefore taking
$\widetilde{X}^n
\in\operatorname{conv}(X^n,\break  X^{n+1}, \dots)$ with the same weights
as $\widetilde{M}^n \in
\operatorname{conv}(M^n, M^{n+1}, \dots)$ gives
\[
\widetilde{X}^k_{\tau} \stackrel{P} {\longrightarrow}
X_{\tau} \qquad\mbox{on $F_n$}
\]
for all $[0,1]$-valued stopping times $\tau$ and for each $n$ and,
since $\widetilde{X}^k = \widetilde{M}^k$ for all $k\geq n$. But,
since $P(F_n^c)<2^{-n}
\to0$, as $n\to\infty$ this implies that $\widetilde{X}^k_{\tau}
\stackrel
{P}{\longrightarrow} X_{\tau}$
for all $[0,1]$-valued stopping times $\tau$. This finishes the proof
in the case when the $X^n$ are local martingales.

For the case of optional strong supermartingales, let $(X^n)^\infty
_{n=1}$ be a sequence of nonnegative optional strong supermartingales
$X^n=(X^n_t)_{0 \leq t \leq1}$ and $X^n = M^n - A^n$ their Mertens
decompositions into a c\`adl\`ag local martingale $M^n$ and a
predictable, nondecreasing, l\`adl\`ag process $A^n$. As the local
martingales $M^n \geq X^n + A^n \geq X^n$ are nonnegative, there
exists by the first part of the proof a sequence of convex combinations
$\widehat{M}^n\in\operatorname{conv}(M^n, M^{n+1},\break \dots)$ and\vspace*{1pt} an
optional strong
supermartingale $\widehat{X}$ with Mertens decomposition $\widehat
{X}= \widehat{M}- \widehat{A}$ such that
%
%e3.33 #&#
\begin{equation}
\label{chris2} \widehat{M}^n_{\tau} \stackrel{P} {
\longrightarrow} \widehat {X}_{\tau}
\end{equation}
for all $[0,1]$-valued stopping times $\tau$. Now let $\widehat{A}^n
\in\operatorname{conv}
(A^n, A^{n+1},\dots)$ be the convex combinations that are obtained with
the same weights as the $\widehat{M}^n$. Then there exists a sequence
$(\widetilde{A}^n)^\infty_{n=1}$ of convex combinations $\widetilde
{A}^n \in
\operatorname{conv}(\widehat{A}^n, \widehat{A}^{n+1}, \dots)$ and a
predictable, nondecreasing, l\`
adl\`ag process $\widetilde{A}$ such that
%
%e3.34 #&#
\begin{equation}
\label{chris1} P \Bigl[\lim_{n\to\infty} \widetilde{A}^n_t
=\widetilde {A}_t, \forall t\in[0,1] \Bigr] =1.
\end{equation}
Indeed, we only need to show that $(\widetilde{A}^n_1)_{n\in\mathbb
{N}}$ is bounded in
$L^0(P)$; then \eqref{chris1} follows from Proposition~3.4 of Campi and
Schachermayer in \cite{CS06}. By monotone convergence we obtain
\[
E\bigl[\widetilde{A}^n_1\bigr] = \lim
_{m\to\infty} E\bigl[\widetilde {A}^n_{1\wedge\sigma^n_m}
\bigr] = \lim_{m\to
\infty} E\bigl[\widetilde{M}^n_{1\wedge\sigma^n_m}
-\widetilde {X}^n_{1\wedge\sigma^n_m}\bigr]\le1
\]
for all $n\in\mathbb{N}$ and therefore the boundedness in $L^0(P)$.
Here $\widetilde{M}^n \in\operatorname{conv}(\widehat{M}^n,\break
\widehat{M}^{n+1}, \ldots)$ and $\widetilde{X}^n \in\operatorname{conv}
(\widehat{X}^n,  \widehat{X}^{n+1}, \dots)$ denote convex
combinations having the same
weights as the $\widehat{A}^n$ and $(\sigma^n_m)^\infty_{m=1}$ is a
localizing
sequence of stopping times for the local martingale~$\widetilde{M}^n$.

Taking convex combinations does not change the convergence \eqref
{chris2}, and so $\widetilde{X}^n \in\operatorname{conv}(X^n,
X^{n+1}, \dots)$ is a sequence
of convex combinations and $\widetilde{X}:=\widehat{X}- \widehat{A}$
an optional strong
supermartingale such that
%
%e3.35 #&#
\begin{equation}
\widetilde{X}^n_{\tau} \stackrel{P} {\longrightarrow}
\widetilde {X}_{\tau}
\end{equation}
for all $[0,1]$-valued stopping times $\tau$.
\end{pf*}

%re3.5 #&#
\begin{remark} (1) Observe that the proof of Theorem~\ref{t1} actually
shows that the limiting optional strong supermartingale $X$ is equal to
$\overline{X}$ up to a set that is included in the graphs of countably many
stopping times $(\tau_m)^\infty_{m=1}$.

(2) Replacing Koml\'os's lemma (Corollary~\ref{kl}) by Koml\'os's
subsequence theorem (Theorem~\ref{kssthm}) in the proof of Theorems
\ref{c1} and \ref{t1}, we obtain, by taking subsequences of subsequences
rather than convex combinations of convex combinations, the following
stronger assertion: Given a sequence $(X^n)^\infty_{n=1}$ of
nonnegative optional strong supermartingales $X^n=(X^n_t)_{0 \leq t
\leq1}$ starting at $X^n_0=1$, there exists a subsequence
$(X^{n_k})^\infty_{k=1}$ and an optional strong supermartingale
$X=(X_t)_{0 \leq t \leq1}$ such that the Ces\`{a}ro means\break $\frac{1}{J}
\sum^J_{j=1} X^{n_{k_j}}$ of any subsubsequence $(X^{n_{k_j}})^\infty
_{j=1}$ converge to $X$ in probability at all finite stopping times, as
$J\to\infty$.
\end{remark}

%s4 #&#
\section{A counterexample}\label{sec4}

At a \emph{single} finite stopping time $\tau$ we may, of course, pass
to a subsequence to obtain that $\widetilde{M}^n_\tau$ converges not
only in probability but also $P$-almost surely to $\widetilde{X}_\tau$.
The next proposition shows that we cannot strengthen Theorem~\ref{c1}
to obtain $P$-almost sure convergence for \emph{all} finite stopping
times simultaneously. The obstacle is, of course, that the set of all
stopping times is far from being countable.

%pr4.1 #&#
\begin{proposition}\label{Ex2}
Let $(M^n)^\infty_{n=1}$ be a sequence of independent nonnegative
continuous martingales $M^n=(M^n_t)_{0 \leq t \leq1}$ starting at
$M^n_0=1$ such that
%
%e4.1 #&#
\begin{equation}\label{Ex2.2}
M^n_\tau\stackrel{P} {\longrightarrow} 1-
\tau
\end{equation}
for all $[0,1]$-valued stopping times $\tau$. Then we have for all
$\varepsilon
>0$ and all sequences $(\widetilde{M}^n)^\infty_{n=1}$ of convex
combinations $\widetilde{M}^n\in\operatorname{conv}(M^n, M^{n
+1},\ldots)$ that there
exists a stopping time $\tau$ such that
\[
P \Bigl[\mathop{\overline{\lim}}_{n\to\infty} \widetilde{M}^n_\tau=+
\infty \Bigr] > 1- \varepsilon.
\]
\end{proposition}

%re4.2 #&#
\begin{remark}
If $(\Omega, \mathcal{F}, (\mathcal{F}_t)_{0 \leq t \leq1}, P)$
supports a sequence $(W^n)^\infty_{n=1}$ of independent Brownian
motions $W^n=(W^n_t)_{0 \leq t \leq1}$, the existence of a sequence
$(M^n)_{n=1}^\infty$ verifying \eqref{Ex2.2} follows similarly as in
the proof of Theorem~\ref{t2} in Section~\ref{sec5} below.
\end{remark}

For the proof of Proposition $\ref{Ex2}$ we will need the following
auxiliary lemma.

%le4.3 #&#
\begin{lemma} \label{Ex2.1}
In the setting of Proposition~\ref{Ex2},
let $\tau$ and $\sigma$ be
two $[0, 1]$-valued stopping times such that $\tau\leq\sigma$ and
$\tau
<\sigma$ on some $A\in\mathcal{F}_\tau$ with $P(A)>0$. Then there
exists, for
all $c>1$, a constant $\gamma=\gamma(c,\tau, \sigma)>0$ and a number
$N=N(\tau,\sigma) \in\mathbb{N}$ such that
\[
P \biggl(\sup_{t\in[\tau,\sigma]} \widetilde{M}^n_t
> c + 1 \biggr) \geq \gamma
\]
for all $n\geq N$.
\end{lemma}

\begin{pf}
Let $\alpha=\frac{E[(\sigma-\tau)\mathbh{1}_A]}{P(A)}$ and
$\varepsilon\in(0,1)$ such that $\alpha> (c+4) \varepsilon$ and
\[
P(B_n) \geq(1-\varepsilon) P (A)
\]
for all $n\geq N$, where
\begin{eqnarray*}
A_n &:=& \bigl\{\bigl| \widetilde{M}^n_\tau- (1-
\tau) \bigr| < \varepsilon\bigr\} \cap A,
\\
B_n &:=& \bigl\{\bigl| \widetilde{M}^n_\sigma- (1-
\sigma) \bigr| < \varepsilon\bigr\} \cap A_n.
\end{eqnarray*}
Then setting $\varrho_n:=\inf\{t \in[\tau, \sigma]|\widetilde
{M}^n_t > c
+ 1\}$ we can estimate
\begin{eqnarray*}
E\bigl[\widetilde{M}^n_\tau\mathbh{1}_{A_n}
\bigr] &=& E\bigl[\widetilde {M}^n_{\varrho
_n\wedge1}
\mathbh{1}_{A_n}\bigr]
\\
&=& E \bigl[\widetilde{M}^n_{\varrho_n\wedge1} (\mathbh{1}_{A_n\cap\{\varrho
_n \leq1\}}
+ \mathbh{1} _{\{\varrho_n>1\}\cap B_n}+\mathbh{1}_{\{
\varrho
_n>1\}\cap B_n^c\cap A_n} ) \bigr]
\\
&\leq & (c + 1) P ( \varrho_n\leq1, A_n) + E\bigl[(1-
\sigma+ \varepsilon) \mathbh{1} _{B_n}\bigr]+ (c + 1) P
\bigl(B_n^c\cap A_n\bigr)
\end{eqnarray*}
by the optional sampling theorem and the continuity of $\widetilde
{M}^n$. Since
\[
E\bigl[\widetilde{M}^n_{\tau} \mathbh{1}_{A_n}
\bigr] \geq E\bigl[(1-\tau -\varepsilon) \mathbh{1}_{A_n}\bigr] \geq E
\bigl[(1-\tau-\varepsilon)\mathbh{1}_{B_n}\bigr],
\]
we obtain that
\begin{eqnarray*}
&& E \bigl[ \bigl((1-\tau-\varepsilon) - (1-\sigma+\varepsilon) \bigr)
\mathbh{1}_{B_n} \bigr] - (c+1) \bigl(P(A)-P(B_n) \bigr)\\
 &&\qquad
\leq  (c + 1) P(\varrho _n\leq 1, A_n)
\\
&&\qquad\leq  (c+1) P(\varrho_n \leq1)
\end{eqnarray*}
and therefore that
\[
\gamma:= \frac{\alpha-3\varepsilon- (c+1)\varepsilon}{c+1} P(A) \leq P(\varrho_n \leq1)=P \biggl(\sup
_{t \in[\tau, \sigma]} \widetilde {M}^n_{\tau
} > c + 1
\biggr)
\]
for all $n\geq N$, where $\gamma>0$ by our choice of $\varepsilon$, as
$E[(\sigma-\tau)\mathbh{1}_{B_n}]\geq(\alpha-\varepsilon)P(A)$.
\end{pf}
\begin{pf*}{Proof of Proposition~\ref{Ex2}}
We shall define $\tau$ as an increasing limit of a sequence of stopping
times $\tau_m$. For this, we set $n_0=0$, $\tau_0=0$ and $\sigma_0=
\frac{1}{2}$ and then define for $m \in\mathbb{N}$ successively
\begin{eqnarray*}
n_m(\omega)&:= & \inf\Bigl\{n\in\mathbb{N}\Big|n>n_{m-1}(
\omega) \mbox{ and $\exists t\in\bigl[\tau _{m-1}(\omega),
\sigma_{m-1} (\omega)\bigr]$}\\
&&\hspace*{140pt}{} \mbox{with $\widetilde {M}^n_t(
\omega) \geq 2^m+1$}\Bigr\},
\\
\tau_m(\omega) &:=& \inf \bigl\{t\in \bigl(\tau_{m-1}(
\omega),\sigma _{m-1}(\omega) \bigr) | \widetilde{M}_t^{n_{m}(\omega)}
(\omega) \geq2^m+1 \bigr\} \wedge1,
\\
\sigma_m(\omega)&:=& \inf\bigl\{t>\tau_m(\omega)|
\widetilde {M}_t^{n_{m}(\omega)} (\omega ) < 2^m \bigr\}
\wedge\sigma_{m-1} (\omega).
\end{eqnarray*}
By construction and the continuity of $\widetilde{M}^n$ we then have,
for all $k\geq m$, that
\[
\widetilde{M}^{n_m(\omega)}_t (\omega) \geq2^m\qquad
\mbox{for all }t\in\bigl[\tau _k(\omega), \sigma_k(\omega)
\bigr]
\]
on $\{\tau_k < 1\}$. Therefore setting $\tau:=\lim_{m\to\infty}
\tau_m$
gives that
\[
\widetilde{M}^{n_m(\omega)}_\tau(\omega) \geq2^m\qquad
\mbox{for all }m
\]
on $\{\tau<1\}$. So it only remains to show that
%
%e4.2 #&#
\begin{equation}\label{eqEX2.1}
P(\tau<1)\geq1-\varepsilon.
\end{equation}
We prove \eqref{eqEX2.1} by induction. For this, assume that there
exists for each $m\in\mathbb{N}_0$, some $\alpha_m>0$ and $N_m\in
\mathbb{N}_0$ such
that $P(D_m)< 1 - \varepsilon2^{-m}$ for
%
%e4.3 #&#
\begin{equation}\label{eqEx2.2}
D_m:=\bigl\{\sigma_m>\tau_m+
\alpha_m,  n_m\in(N_{m-1}, N_m]\bigr\}.
\end{equation}
Indeed, for $m=0$, we can choose $\alpha_0=\frac{1}{2}, N_{-1}=0$ and $N_0=1$.
Regarding the induction step we first show that $n_m<\infty$
$P$-a.s. on $D_{m-1}$. To that end, we can assume w.l.o.g. that the
$ (\widetilde{M}^n )_{n=1}^\infty$ are also independent by
choosing the blocks of which we take the convex combinations disjoint
and passing to a subsequence. As we are only making an assertion about
the limes superior, this will be sufficient. Moreover, we observe that
\[
F:=\{n_m<\infty\}\cap D_{m-1}=\bigcup
_{n=N_{m-1}}^\infty F_n\cap D_{m-1}
\]
with $F_n:= \{\exists t\in (\tau_{m-1} (\omega), \sigma
_{m-1}(\omega
) ]| \widetilde{M}^n_t(\omega)\geq2^m+1 \}$. Then
using the
estimate $1-x\leq\exp(-x)$ and the independence of the $F_n$ of each
other and $D_{m-1}$ gives
\begin{eqnarray*}
P\bigl(D_{m-1}\cap F^c\bigr)&=& \lim_{k\to\infty}
P \Biggl(\bigcap_{n=N_{m-1}}^k
F_n^c \Biggr) P(D_{m-1})
\\
&=&\lim_{k\to\infty} \prod^k_{n=N_{m-1}}
\bigl(1-P(F_n) \bigr) P(D_{m-1})
\\
&\leq & \lim_{k\to\infty} \exp \Biggl(-\sum
^k_{n=N_{m-1}} P(F_n) \Biggr)
P(D_{m-1}).
\end{eqnarray*}
Since $\sum^{\infty}_{n=N_{m-1}} P(F_n)=\infty$ by Lemma~\ref{Ex2.1},
this implies that $P(D_{m-1} \cap F^c) = 0$ and hence that $n_m<\infty$
$P$-a.s. on $D_{m-1}$. More precisely,\vspace*{1pt} by applying Lemma~\ref{Ex2.1}
for $c=2^m$ with $\tau=\tau_{m-1}$, $\sigma=\sigma_{m-1}$ and
$A=D_{m-1}$ to $\widetilde{M}^n$ for $n\geq N_{m-1}$, we get that
$P(F_n)\geq
\gamma>0$ for all $n\geq N_{m-1}$. Therefore $\tau_m < 1$ $P$-a.s. on
$D_{m-1}$ as well. By the continuity of the $\widetilde{M}^n$ and, as
$\tau_m<\frac{1}{2}$ on $D_{m-1}$, we obtain that $\frac{1}{2} \geq
\sigma_m>\tau_m$ $P$-a.s. on $D_{m-1}$, which finishes the induction step.

Now, since $\{\tau<1\}\supseteq\bigcap_{m=1}^\infty D_m=:D$ and
\[
P(D)\geq1- \sum_{m=1}^\infty P
\bigl(D_m^c\bigr)= 1- \sum_{m=1}^\infty
\frac
{\varepsilon
}{2^m}=1-\varepsilon,
\]
we have established \eqref{eqEx2.2}, which completes the proof of the
proposition.
\end{pf*}
%
%s5 #&#
\section{Proof of Theorem \texorpdfstring{\protect\ref{t2}}{2.8}}\label{sec5}
We now pass to the proof of Theorem~\ref{t2}. The following lemma
yields a building\vspace*{-6pt} block.

%le5.1 #&#
\begin{lemma} \label{lt2}
Let $W=(W_t) _{0 \leq t \leq1}$ be a standard Brownian motion on
$(\Omega, \mathcal{F}, \mathbb{F}, P)$ and $\varrho$ a $[0,1] \cup\{\infty
\}
$-valued stopping time. Then there exists a sequence $(\varphi
^n)^{\infty}_{n=1}$ of predictable integrands of finite variation such
that $M^n:= \varphi^n \bolds{\cdot}W \geq-1$
is a bounded martingale for each
$n\in\mathbb{N}$\vspace*{-2pt} and
%
%e5.1 #&#
\begin{equation}
M^n_\tau\mathop{\longrightarrow}^{P\mbox{-}\mathit{a.s.}} -
\mathbh{1}_{\rrbracket
\varrho, 1
\rrbracket} (\tau)=-\mathbh{1}_{\{\tau> \varrho\}},\qquad \mbox{as }n\to
\infty,
\end{equation}
for all $[0, 1]$-valued stopping times\vspace*{-6pt} $\tau$.
\end{lemma}

\begin{pf}
We consider the case $\varrho\equiv0$ first. There are many possible
choices for the integrands $(\varphi^n)^{\infty}_{n=1}$. To come up
with one, we use the deterministic\vspace*{-1pt} functions
\[
\psi^n_t:= \frac{1}{2^{-n}-t}\mathbh{1}_{(0,2^{-n})}(t).
\]
Then the continuous martingales $N^n:= (\psi^n \bolds{\cdot} W_t)_{0 \leq t
<2^{-n}}$ are well defined, for each $n \in\mathbb{N}$. It follows
from the Dambis--Dubins--Schwarz theorem that the stopping times\vspace*{-2pt}
\begin{eqnarray*}
\tau_n &:=& \inf\bigl\{t \in\bigl(0,2^{-n}\bigr) |
N^n_t= -1\bigr\},
\\
\sigma_{n,k} &:=& \inf\bigl\{t \in\bigl(0,2^{-n}\bigr) |
N^n_t >k\bigr\}
\end{eqnarray*}
are $P$-a.s. strictly smaller than $2^{-n}$ for all $n,k\in\mathbb
{N}$,\vspace*{-1pt} since
\[
\bigl\langle N^n\bigr\rangle_t= \frac{1}{2^{-n}-t} -
\frac{1}{2^{-n}} \qquad\mbox{for } t\in[0,2^{-n})
\]
and $\lim_{t\nearrow2^{-n}}\langle N^n\rangle_t=\infty$. Therefore setting
$\widetilde{\psi}^{n,k} = \psi^n \mathbh{1}_{\llbracket0,\tau
_n\wedge
\sigma_{n,k} \rrbracket}$
gives a sequence\vspace*{-2pt}
\[
\widetilde{N}^{n,k} = \widetilde{\psi}^{n,k} \bolds{\cdot} W = \bigl(\psi^n \bolds{\cdot}
W\bigr)^{\tau_n\wedge\sigma_{n,k}}
\]
of bounded martingales such that, for all $[0, 1]$-valued stopping
times\vspace*{-1pt} $\tau$,
\[
\widetilde{N}^{n,k}_{\tau}\mathop{\longrightarrow}^{P\mbox{-}\mathrm{a.s.}}
-1 \qquad\mbox{on }\bigl\{\tau\geq2^{-n}\bigr\}, \mbox{ as } k\to\infty,
\]
since $\sigma_{n,k} \nearrow 2^{-n}$ $P$-a.s, as $k\to\infty$.
Defining $\varphi^n:=\widetilde{\psi}^{n,k(n)}$ and $M^n=\widetilde
{N}^{n,k(n)}$ as a suitable diagonal sequence such\vspace*{-1pt} that
$M^n_{2^{-n}}=\widetilde{N}^{n,k(n)}_{2^{-n}}\to-1$, as $n\to\infty$,
then yields the assertion for $\varrho\equiv0$, as $M^n_0=0$ for all
$n\in\mathbb{N}
$ and $\mathbh{1}_{\{\tau\geq2 ^{-n}\}} \stackrel{P\mbox{-}\mathrm{a.s.}}{\longrightarrow}
\mathbh{1}_{\{\tau> 0\}}$, as $n\to\infty$.

Next we observe that if we consider for some $[0, 1] \cup\{\infty\}
$-valued stopping time~$\sigma$ the stopped Brownian notion $W^{\sigma
}=(W_{\sigma\wedge t})_{0 \leq t \leq1}$, then we obtain by the above
argument\vspace*{-2pt} that
\[
\bigl(M^n\bigr)^{\sigma}_{\tau} =
M^n_{\sigma\wedge\tau} = \bigl(\varphi^n \stackrel{\mbox{
\tiny$\bullet$}} {} \bigl(W^{\sigma}\bigr) \bigr)_{\tau} \mathop{
\longrightarrow}^{P\mbox{-}\mathrm{a.s.}} \mathbh{1}_{(0,1)} (\sigma\wedge\tau)
\]
for every $[0,1]$-valued stopping time $\tau$.

For the general case $\varrho\not\equiv0$, consider the process
$\overline
{W}_t:=(W_{t+\varrho} - W_\varrho)_{0 \leq t \leq1}$ which is a Brownian
motion with respect to the filtration $\overline{\mathbb
{F}}:=(\overline
{\mathcal{F}}_t)_{0 \leq t \leq1}:=(\mathcal{F}_{(t+\varrho)\wedge1})_{0
\leq t \leq1}$ that is independent of $\mathcal{F}_{\varrho}$ and stopped
at the $\overline{\mathbb{F}}$-stopping time $\bar{\sigma
}:=(1-\varrho)$. Then the
general case $\varrho\not\equiv0$ follows by applying the result for
$\varrho
\equiv0$ for the stopped Brownian motion $\overline{W}$ and the
stopping time $\bar{\tau} = (\tau- \varrho)_{\{\tau> \varrho\}}$
which is
always smaller than $\bar{\sigma}$. Indeed, as the corresponding
martingales $\overline{M}^n$ obtained for $\overline{W}$ with respect
to $(\overline{\mathcal{F}}_t)_{0 \leq t \leq1}$ start at $0$, the processes
\[
M^n_t(\omega) =
\cases{ 0, &\quad$t \leq
\varrho(\omega) \wedge1$,\vspace*{3pt}
\cr
\overline{M}^n_{t+\varrho(\omega)}(
\omega), &\quad$\varrho(\omega) < t \leq1$,}
\]
are martingales with respect to the filtration $\mathbb{F}=(\mathcal{F}_t)_{0
\leq t \leq1}$ that converge to $\mathbh{1}_{\llbracket\varrho, 1
\rrbracket} (\tau)$ $P$-a.s. for every $[0,1]$-valued $\mathbb{F}$-stopping
time $\tau$.
\end{pf}

\begin{pf*}{Proof of Theorem~\ref{t2}}
Let $X=M - A$ be the Mertens decomposition of the optional strong
supermartingale $X$. It is then sufficient to show the assertion for
$M$ and $A$ separately.
\begin{longlist}[(2)]
\item[(1)] We begin with the local martingale $M$. As any localizing sequence
$(\tau_m)^\infty_{m=1}$ of stopping times for $M$ gives a sequence
$\widetilde{M}^m:= M^{\tau_m}$ of martingales that converges uniformly
in probability, we obtain a sequence $\overline{M}^n$ of martingales
that converges $P$-a.s. uniformly to $M$ by passing to a subsequence
$(\widetilde{M})^\infty_{n=1}$ such that $P(\tau_n <1) < 2^{-n}$. To
see that we can choose the $M^n$ to be bounded, we observe that setting
\[
\overline{M}_t^{n,k}:= E \bigl[\overline{M}^n_1
\wedge k \vee- k | \mathcal{F}_t\bigr]
\]
for $t\in[0,1]$ gives\vspace*{-1pt} for every martingale $\overline{M}^n$ a sequence
of bounded martingales $\overline{M}^{n,k} = (\overline{M}^{n,k}_t)_{0
\leq t \leq1}$ such that $\overline{M}^{n,k}_1\stackrel{L^1(P)}{\longrightarrow}\overline{M}^n_1$, as $k\to\infty$, and therefore locally in
$\mathcal{H}^1(P)$ by Theorem~4.2.1 in \cite{J79}. By the
Burkholder--Davis--Gundy inequality (see, e.g., Theorem IV.48 in \cite
{P04}), this also implies uniform convergence in probability and hence
$P$-a.s. uniform convergence by passing to a subsequence, again indexed
by $k$. Then taking a diagonal sequence $(\overline
{M}^{n,k(n)})^\infty
_{n=1}$ gives a sequence of martingales $(M^n)^\infty_{n=1} =
(\overline
{M}^{n,k(n)})^\infty_{n=1}$ that converges $P$-a.s. uniformly to $M$
and therefore also satisfies \eqref{eqt2} for every $[0,1]$-valued
stopping time $\tau$.

\item[(2)] To prove the assertion for the predictable part $A$, we decompose
\[
A=A^{c}+\sum^\infty_{i=1}
\Delta_{+} A_{\sigma_i} \mathbh{1}_{\rrbracket\sigma_i, 1 \rrbracket}+\sum
^\infty_{j=1} \Delta A_{\varrho
_j}
\mathbh{1}_{\llbracket\varrho_j, 1 \rrbracket}
\]
into its continuous part $A^c$, its totally right-discontinuous part\vspace*{1pt}
$A^{rd}:=\break  \sum^\infty_{i=1} \Delta_{+} A_{\sigma_i} \mathbh{1}_{\rrbracket\sigma_i, 1 \rrbracket}$ and totally left-discontinuous
part $A^{ld}:=\sum^\infty_{j=1} \Delta A_{\varrho_j} \mathbh{1}_{\llbracket\varrho_j, 1 \rrbracket}$.\vspace*{1pt} By superposition it is
sufficient to approximate $-A^c$, each single right jump process
$-A_{\sigma_i} \mathbh{1}_{\rrbracket\sigma_i, 1 \rrbracket}$ for
$i\in\mathbb{N}$ and each single left jump process $-\Delta
A_{\varrho_j} \mathbh{1}_{\llbracket\varrho_j, 1 \rrbracket}$ for $j\in\mathbb{N}$
separately. Indeed,
let $(M^{c,n})_{n=1}^\infty$, $(M^{rd,i,n})_{n=1}^\infty$ for each
$i\in
\mathbb{N}$ and $(M^{ld,j,n})_{n=1}^\infty$ for each $j\in\mathbb
{N}$ be sequences of
bounded martingales such that
%
%e5.2 #&#
%e5.3 #&#
%e5.4 #&#
\begin{eqnarray}
\label{peq1}
M^{c,n}_{\tau}&\stackrel{P}\longrightarrow& -A^{c}_\tau,
\\
\label{peq2}
M^{rd,i,n}_{\tau}&\stackrel{P}\longrightarrow & -\Delta_{+} A_{\sigma_i} \mathbh{1}_{\rrbracket\sigma_i, 1 \rrbracket}(
\tau),
\\
\label{peq3}
M^{ld,j,n}_{\tau}&\stackrel{P}\longrightarrow & -\Delta
A_{\varrho_j} \mathbh{1}_{\llbracket\varrho_j, 1 \rrbracket}(\tau),
\end{eqnarray}
as $n\to\infty$, for all $[0,1]$-valued stopping times $\tau$. Then setting
\[
M^n:=M^{c,n}+\sum_{i=1}^n
M^{rd,i,n}+\sum_{j=1}^n
M^{ld,j,n}
\]
gives a sequence of bounded martingales such that $M^n_{\tau}\stackrel
{P}\longrightarrow-A_\tau$, as $n\to\infty$, for all $[0,1]$-valued
stopping times $\tau$. %Therefore it only remains to prove the
%existence of $(M^{c,n})_{n=1}^\infty$, $(M^{ld,i,n})_{n=1}^\infty$ and
%$(M^{ld,i,n})_{n=1}^\infty$ such that \eqref{peq1}, \eqref{peq2} and
%\eqref{peq3}.
\begin{longlist}[(2a)]
\item[(2a)] We begin with showing the existence of $(M^{rd,i,n})_{n=1}^\infty
$ for some fixed $i\in\mathbb{N}$. For this, we set
\[
\vartheta_t^{i,n}:= (\Delta_+ A_{\sigma_i} \wedge n)
\mathbh{1} _{\rrbracket
\sigma_i, 1 \rrbracket} \varphi^n_{t} \in
L^2(W),
\]
where $(\varphi^n)^\infty_{n=1}$ is a sequence of integrands as
obtained in Lemma~\ref{lt2} for the stopping time $\varrho=\sigma
_i$. Then
it follows immediately from Lemma~\ref{lt2} that $\vartheta^{i,n}
\bolds{\cdot}
W_\tau\stackrel{P\mbox{-}\mathrm{a.s.}}{\longrightarrow} \Delta_+ A_{\sigma_i} \mathbh{1}_{\rrbracket\sigma_i, 1 \rrbracket} (\tau)$, as $n\to\infty$, for
every $[0,1]$-valued stopping time $\tau$ and therefore that
\[
M^{rd,i,n}:= \vartheta^{i,n} \bolds{\cdot} W
\]
gives a sequence of bounded martingales such that \eqref{peq2} holds.
Note that by the construction of the integrands $\varphi^n$ in Lemma~\ref
{lt2} the approximating martingales $M^{rd,i,n}$ are $0$ on
$\llbracket0,\sigma_i\rrbracket$, constant to either $-\Delta_+
A_{\sigma_i} \wedge n$ or $(\Delta_+ A_{\sigma_i} \wedge n) k(n)$ on
$\llbracket\sigma_i+2^{-n},1\rrbracket$. Therefore they converge
$P$-a.s. uniformly to $-\Delta_+ A_{\sigma_i}$ on $\llbracket\sigma
_i+2^{-m},1\rrbracket$ for each $m\in\mathbb{N}$.

\item[(2b)]  To obtain the approximating sequence $(M^{ld,j,n})_{n=1}^\infty$
for some fixed $j\in\mathbb{N}$, we observe that the stopping time
$\varrho_j$ is
predictable and let $(\varrho_{j,k})^\infty_{k=1}$ be an announcing
sequence of stopping times, that is, a nondecreasing sequence of
stopping times such that $\varrho_{j,k} < \varrho_j$ on $\{\varrho_j
> 0\}$ and $\varrho
_{j,k} \stackrel{P\mbox{-}\mathrm{a.s.}}{\longrightarrow} \varrho_j$, as $k\to\infty$. Since
$\Delta A_{\varrho_j}\in L^1(P)$ is $\mathcal{F}_{\varrho
_j-}$-measurable by
Theorem\vspace*{1pt} IV.67.b) in \cite{DM78} and $\mathcal{F}_{\varrho
_j-}=\bigvee^\infty_{k=1} \mathcal{F}_{\varrho_{j,k}}$ by Theorem IV.56.d) in \cite
{DM78}, we
have that
%
%e5.5 #&#
\begin{equation}
E[\Delta A_{\varrho_j} | \mathcal{F}_{\varrho_{j,k}}] \mathop{\longrightarrow}^{P\mbox{-}\mathrm{a.s.}} \Delta A_{\varrho_j},\qquad \mbox{as } k\to\infty,
\end{equation}
by martingale convergence. Therefore setting
%
%e5.6 #&#
\begin{equation}
\widetilde{A}^{ld,j,k}:= E[\Delta A_{\varrho_j} | \mathcal
{F}_{\varrho_{j,k}} ] \mathbh{1}_{\rrbracket\varrho_{j,k},1 \rrbracket}
\end{equation}
gives a sequence\vspace*{-2pt} of single right jump processes that converges to
$\Delta A_{\varrho_j} \mathbh{1}_{\llbracket\varrho_j, 1
\rrbracket}$
$P$-a.s. at each $[0,1]$-valued stopping time $\tau$, since $\mathbh{1}_{\rrbracket\varrho_{j,k,},1 \rrbracket} (\tau)
\stackrel{P\mbox{-}\mathrm{a.s.}}{\longrightarrow} \mathbh{1}_{\llbracket\varrho_j,1 \rrbracket}
(\tau)$,
as $k\to
\infty$, for all $[0,1]$-valued stopping times $\tau$.

By part (2a) there exists for each $k\in\mathbb{N}$ a sequence
$(\widetilde
{M}^{j,k,n})_{n=1}^\infty$ of bounded martingales such that
$\widetilde
{M}^{j,k,n}_\tau\stackrel{P\mbox{-}\mathrm{a.s.}}{\longrightarrow}-\widetilde
{A}^{ld,j,k}_\tau
$, as $n\to\infty$, for all $[0,1]$-valued stopping times $\tau$. For
the stopping time $\varrho_j$ we can therefore find a diagonal sequence
$(\widetilde{M}^{j,k,n(k)})_{k=1}^\infty$ such that $\widetilde
{M}^{j,k,n(k)}_{\varrho_j} \stackrel{P\mbox{-}\mathrm{a.s.}}{\longrightarrow}-\widetilde
{A}^{ld,j,k}_{\varrho_j}$, as $k\to\infty$. By the proof of Lemma~\ref
{lt2} and part (2a) above we can choose the martingales $\widetilde
{M}^{j,k,n(k)}$ such that $\widetilde{M}^{j,k,n(k)}\equiv0$ on
$\llbracket0,\varrho_{j,k}\rrbracket$ and $\widetilde
{M}^{j,k,n(k)}\equiv
- (E[\Delta A_{\varrho_j} | \mathcal{F}_{\varrho_{j,k}} ]\wedge
n(k) )$
on $\llbracket(\varrho_{j,k}+2^{-n(k)})_{F_k},1\rrbracket$, where
the set
\[
F_k:= \Bigl\{\widetilde{M}^{j,k,n(k)}_{\varrho_j+2^{-n(k)}}= -
\bigl(E[\Delta A_{\varrho_j} | \mathcal{F}_{\varrho_{j,k}} ]\wedge n(k) \bigr)
\Bigr\}
\]
has probability $P(F_k)>1-2^{-k}$. This sequence $(\widetilde
{M}^{j,k,n(k)})_{k=1}^\infty$ therefore already satisfies $\widetilde
{M}^{j,k,n(k)}_\tau\stackrel{P\mbox{-}\mathrm{a.s.}}{\longrightarrow}-\Delta A_{\varrho_j}
\mathbh{1}_{\llbracket\varrho_j, 1 \rrbracket}(\tau)$ for all
$[0,1]$-valued stopping times $\tau$ and we have \eqref{peq3}.

\item[(2c)]  For the approximation of the continuous part $A^c$, we observe
that by the left-continuity and adaptedness of $A^c$ there exists a
sequence $(\widetilde{A}^n)^\infty_{n=1}$ of nondecreasing integrable
simple predictable processes that converges uniformly in probability to
$A^c$ and hence $P$-a.s. uniform by passing to a fast convergent
subsequence again indexed by $n$; see for example Theorem II.10 in
\cite{P04}. Recall that a simple predictable process is a predictable
process $\widetilde{A}$ of the form
%
%e5.7 #&#
\begin{equation}
\label{simple} \widetilde{A}=\sum^m_{i=1}
\Delta_{+} A_{\sigma_i} \mathbh{1}_{\rrbracket\sigma_i, 1 \rrbracket},
\end{equation}
where $(\sigma_i)^m_{i=1}$ are $[0,1]\cup\{\infty\}$-valued stopping
times such that $\sigma_i < \sigma_{i+1}$ for $i=1, \dots, m-1$ and
$\Delta_+ A_{\sigma_i}$ is $\mathcal{F}_{\sigma_i}$-measurable.

By part (2a) there\vspace*{-1.5pt} exists, for each $n\in\mathbb{N}$, a sequence
$(\widetilde
{M}^{n,k})_{k=1}^\infty$ of martingales such that $\widetilde
{M}^{n,k}_\tau\stackrel{P\mbox{-}\mathrm{a.s.}}{\longrightarrow}-\widetilde{A}^n_\tau$, as
$k\to\infty$, for all $[0,1]$-valued stopping times $\tau$. Therefore
we can pass to a diagonal sequence $\widetilde{M}^{n,k(n)}$ such that
%
%e5.8 #&#
\begin{equation}
P \Bigl[\lim_{n\to\infty}\widetilde{M}^{n,k(n)}_q=-A^c_q,
 \forall q\in \mathbb{Q}\cap[0,1] \Bigr]=1.\label{peq4}
\end{equation}
By Theorem~\ref{t1} there exists a sequence $(M^n)_{n=1}^\infty$ of
convex combinations
\[
M^n\in\operatorname{conv}\bigl(\widetilde{M}^{n,k(n)},
\widetilde {M}^{n+1,k(n+1)},\ldots\bigr)
\]
and an optional strong supermartingale $X$ such that $M^n_\tau
\stackrel{P}{\longrightarrow} X_\tau$ for all $[0,1]$-valued stopping times $\tau$.

To complete the proof it therefore only remains to show that $X=-A^c$.
For this, we argue by contradiction and assume that the optional set
$G:=\{X\ne-A^c\}$ is not evanescent, that is, that $P (\pi(G)
)>0$, where $\pi ((\omega,t) )=\omega$ denotes the
projection on the first
component. By the optional cross-section theorem (Theorem IV.84 in
\cite{DM82}) there then exists a $[0,1]\cup\{\infty\}$-valued stopping time
$\tau$ such that $X_\tau\ne-A^c_\tau$ on $F:=\{\tau<\infty\}$ with
$P(F)>0$, which we can decompose into an accessible stopping time $\tau
^A$ and a totally inaccessible stopping time $\tau^I$ such that $\tau
=\tau^A\wedge\tau^I$ by Theorem IV.81.c) in \cite{DM78}. On $\{\tau
^I<\infty\}$ we obtain that $M^n_{\tau^I-}=M^n_{\tau^I}\stackrel{P}{\longrightarrow} X_{\tau^I}$ and $A^c_{\tau^I-}=A^c_{\tau^I}$ from\vspace*{-1pt}
the continuity of $M^n$ and $A^c$. Therefore $X_{\tau^I}=-A^c_{\tau
^I}$, as $M^n_{\tau^I-}\stackrel{P}{\longrightarrow} X_{\tau^I-}$ by
Proposition~\ref{propti} and $X_{\tau^I-}=-A^c_{\tau^I-}$ by \eqref{peq4}. This implies that $P(\tau^I<\infty)=0$ and hence $P(\tau
^A<\infty)=P(F)>0$. Since $\tau^A$ is accessible, there exists a
predictable stopping time $\sigma$ such that $P(\tau^A=\sigma<\infty
)>0$. By the strong supermartingale property of $X$ we have that
\[
X_{\sigma-}\geq E[X_{\sigma}|\mathcal{F}_{\sigma-}]\geq
E[X_{\sigma+}|\mathcal{F} _{\sigma-}]\qquad \mbox{on } \{\sigma<\infty\},
\]
as $\sigma$ is predictable. Since $X_-=-A^c_-$ and $X_+=-A^c_+$ by
\eqref{peq4}, this implies that $X_{\sigma}=-A^c_{\sigma}$ by the
continuity of $A^c$. However, this contradicts $P(F)>0$ and therefore
shows \eqref{peq1}, which completes the proof.\quad\qed
\end{longlist}
\end{longlist}
\noqed\end{pf*}
%
%s6 #&#
\section{Proof of Theorem \texorpdfstring{\protect\ref{t3}}{2.11}} \label{sec6}
We begin with the proof of Proposition~\ref{propti}, and for this, we
will use the following variant of Doob's up-crossing inequality that
holds uniformly over the set $\mathfrak{X}$ of nonnegative optional
strong supermartingales $X=(X_t)_{0 \leq t \leq1}$ starting at $X_0=1$.

%le6.1 #&#
\begin{lemma}\label{lW1}
For each $\varepsilon>0$ and $\delta>0$, there exists a constant
$C=C(\varepsilon,\delta) \in\mathbb{N}$ such that
\[
\sup_{X \in\mathfrak{X}} P \bigl[M_{\varepsilon} (X) > C\bigr]< \delta,
\]
where the random variable $M_\varepsilon(X)$ is pathwise defined as the
maximal amount of moves of the process $X$ of size bigger than
$\varepsilon$, that is,
\begin{eqnarray*}
&& M_\varepsilon(X) (\omega) \\
&& \qquad:= \sup \Bigl\{m \in\mathbb{N} \Big|
\bigl|X_{t_i}(\omega) - X_{t_{i-1}}(\omega)\bigr| > \varepsilon,\mbox{ for }
0\leq t_0 < t_1 < \cdots< t_m \leq1 \Bigr\}.
\end{eqnarray*}
\end{lemma}

\begin{pf}
Choose $n \in\mathbb{N}$ such that $\frac{1}{n} \leq\frac
{\varepsilon
}{2}$, fix some $X\in\mathfrak{X}$ and denote by $X=M-A$ its Mertens
decomposition. Then $M=X+A$ is a nonnegative c\`adl\`ag local
martingale and hence a c\`adl\`ag supermartingale such that
\[
E[M_t]\leq1
\]
for all $t\in[0,1]$. Letting $C_1 \in\mathbb{N}$ with $C_1 \geq
\frac
{2}{\delta}$ we obtain from Doob's maximal inequality that
\[
P \Bigl(M_1^*:=\sup_{0 \leq s \leq1} M_s >
C_1 \Bigr) \leq\frac
{1}{C_1} \leq\frac{\delta}{2}.
\]
Then we divide the interval $[0, C_1]$ into $n C_1=:N$ subintervals
$I_k:= [\frac{k}{N}, \frac{k+1}{N}]$ of equal length of at most
$\frac
{\varepsilon}{2}$ for $k=0, \dots, N-1$. The basic intuition behind
this is that whenever the nonnegative (c\`adl\`ag) local martingale
$M=(M_t)_{0 \leq t \leq1}$ moves more than $\varepsilon$, while its
supremum stays below $C_1$, it has at least to cross one of the
subintervals $I_k$.
For each interval $I_k$ we can estimate the number $U (M; I_k)$ of
up-crossings of the interval $I_k$ by the process $M=(M_t)_{0 \leq t
\leq1}$ up to time $1$ by Doob's up-crossing inequality by
\[
P\bigl[U (M;I_k) > C_2\bigr] \leq\frac{N}{C_2} E
\bigl[U (M;I_k)\bigr] \leq\frac{N}{C_2} \sup_{0 \leq t \leq1}
E[M_t] \leq\frac{N}{C_2}.
\]
Choosing $\tilde{C}_2=\frac{2N^2}{\delta}$ we obtain that
\[
P\bigl[U(M;I_k) >\tilde{C}_2\bigr] \leq
\frac{\delta}{2N}.
\]
Then summing over all intervals gives for the number $U_{\varepsilon
}(M)$ of up-moves of the process $M$ of size $\varepsilon$ that
\begin{eqnarray*}
&& P\bigl[U_\varepsilon(M) >\tilde C_2N\bigr]\\
&&\qquad\leq   P
\bigl[M^*_1 \leq C_1, \exists k\in\{1, \dots, N\}
\mbox{ with }U(M;I_k) > \tilde{C}_2\bigr] + P
\bigl[M^*_1 >C_1\bigr]
  \leq  \delta.
\end{eqnarray*}
Since $X=M-A$ is nonnegative starting at $X_0=1$ and $A$ is
nondecreasing, the number $M_\varepsilon(X)$ of moves of $X$ of size
$\varepsilon$ is smaller than $2(U_{\varepsilon}(X) + N)$. Therefore we
can conclude that
%
%e6.1 #&#
\begin{equation}
\label{lW1eq1}
P\bigl[M_{\varepsilon}(X) > C\bigr] \leq\delta
\end{equation}
for $C=2(\tilde{C}_2 +1) N$. To\vspace*{1pt} complete the proof, we observe that the
constants $C_1$ and $C=2(\tilde{C}_2 + 1)N$ are independent of the
choice of the optional strong supermartingale $X\in\mathfrak{X}$, and
we can therefore take the supremum over all $X\in\mathfrak{X}$ in the
inequality (\ref{lW1eq1}).
\end{pf}

Let $X= (X_t)_{0 \leq t\leq1}$ be a l\`ag (existence of left limits)
process and $\tau$ be a $(0,1]$-valued stopping time. For $m\in
\mathbb{N}$,
let $\tau_m$ be the $m$th dyadic\vspace*{1pt} approximation of the stopping time
$\tau$ %given by
%$$\tau_m:=\inf\{t\in D_m~|~t > \tau\},  \mbox{where}   D_m=
%\{k2^{-m}|k=1, \dots, 2^m\},$$
%which is a $\{\frac{1}{2^m}, \dots,1\}$-valued stopping time.
as defined in \eqref{A1.1}. Note that $\tau_m$ is $\{\frac{1}{2^m},
\dots,1\}$-valued, as $\tau>0$. As $(X_t)_{0 \leq t \leq1}$ is assured
to have l\`ag trajectories, we obtain
%
%e6.2 #&#
\begin{equation}
\label{P2}
X_{\tau_m-2^{-m}} \mathop{\longrightarrow}^{P\mbox{-}\mathrm{a.s.}}
X_{\tau-},\qquad \mbox{as } m\to\infty,
\end{equation}
and therefore in probability. The next lemma gives a quantitative
version of this rather obvious fact.

%le6.2 #&#
\begin{lemma} \label{lW2}
Let $\tau$ be a totally inaccessible $(0,1]$-valued stopping time. Then
the convergence in  \eqref{P2} above holds true in probability  uniformly
over all nonnegative optional strong supermartingales $X \in\mathfrak
{X}$, that is, $X=(X_t)_{0 \leq t \leq1}$, starting at $X_0=1$. More
precisely, we have for each $\varepsilon>0$ that
%
%e6.3 #&#
\begin{equation}
\label{P3a}
\lim_{m\to\infty} \sup_{X \in\mathfrak{X}} P\bigl[|
X_{\tau_m-2^{-m}} -X_{\tau-} | >\varepsilon\bigr]=0.
\end{equation}
\end{lemma}

\begin{pf}
Denote by $A=(A_t)_{0 \leq t \leq1}$ the compensator of $\tau$, which
is the unique continuous increasing process such that $(\mathbh{1}_{\llbracket\tau,1 \rrbracket} - A_t)_{0 \leq t \leq1}$ is a martingale.
For every predictable set $G\subseteq\Omega\times[0,1]$, we then have
%
%e6.4 #&#
\begin{equation}
\label{186}
\hspace*{13pt}\quad P [\tau\in G ] = E [\mathbh{1}_G
\mathbh{1}_{\llbracket\tau\rrbracket} ] = E \biggl[\int^1_0
\mathbh{1}_G(t)\,d \mathbh{1}_{\llbracket\tau, 1 \rrbracket} (t) \biggr] = E \biggl[
\int^1_0\mathbh{1}_G(t)\,dA_t
\biggr].
\end{equation}
Here we used that the predictable $\sigma$-algebra on $\Omega\times
[0,1]$ is generated by the left-open stochastic intervals, that is,
intervals of the form $\rrbracket\sigma_1, \sigma_2\rrbracket$ for
stopping times $\sigma_1$ and $\sigma_2$ and a monotone class argument
to deduce the second equality in \eqref{186}. The third equality is the
definition of the compensator.
Fix $X \in\mathfrak{X}$, $\varepsilon>0$, $\delta>0$ and apply Lemma~\ref{lW1} and the integrability of $A_1$ to find $c=c(\varepsilon,
\delta, \tau)$ such that the exceptional set
%
%e6.5 #&#
\begin{equation}
\label{187}
F_1=\bigl\{M_{\varepsilon} (X) \geq c\bigr\}
\end{equation}
satisfies
%
%e6.6 #&#
\begin{equation}
\label{188}
E[\mathbh{1}_{F_1} A_1]<\delta.
\end{equation}
Find $m$ large enough such that
%
%e6.7 #&#
\begin{equation}
\label{189} E[\mathbh{1}_{F_2} A_1]<\delta,
\end{equation}
where $F_2$ is the exceptional set
%
%e6.8 #&#
\begin{equation}
\label{190}
F_2= \biggl\{\exists k\in\bigl\{1,
\ldots,2^m\bigr\}\mbox{ such that }A_{{k}/{2^m}} - A_{({k-1})/{2^m}} >
\frac{\delta}{c} \biggr\}.
\end{equation}
Define $G$ to be the predictable set
%
%e6.9 #&#
\begin{eqnarray}
G&=& \bigcup^{2^m}_{k=1} \biggl
\{(\omega, t)  \bigg| \frac{k-1}{2^m} < t \leq\frac{k}{2^m} \mbox{ and}
\nonumber
\\[-8pt]
\label{191}\\[-8pt]
\nonumber
 &&\hspace*{50pt}{}\sup
_{({k-1})/{2^m}
\leq u
\leq t} \bigl| X_{u-}(\omega) - X_{({k-1})/{2^m}} (\omega)
\bigr| \leq \varepsilon \biggr\}.
\end{eqnarray}
We then have $P [\tau\notin G] < 3 \delta$. Indeed, applying \eqref
{186} to the complement $G^c$ of $G$ we get
\[
P [\tau\notin G] = E \biggl[ (\mathbh{1}_{F_1 \cup F_2} + \mathbh{1}_{\Omega\setminus(F_1 \cup F_2)}
) \int^1_0 \mathbh{1}_{G^c
}\,dA_t
\biggr],
\]
where $F_1$ and $F_2$ denote the exceptional sets in \eqref{187} and
\eqref{190}. By \eqref{188} and \eqref{189},
%
%e6.10 #&#
\begin{equation}
E \biggl[\mathbh{1}_{F_1 \cup F_2} \int^1_0
\mathbh{1}_{G^c} \,dA_t \biggr] \leq2\delta.
\end{equation}
On the set $\Omega\setminus(F_1 \cup F_2)$ we deduce from \eqref{187},  \eqref{190} and \eqref{191} that
\[
\int^1_0 \mathbh{1}_{G^c}\,dA_t
\leq c \frac{\delta}{c} = \delta
\]
so that
%
%e6.11 #&#
\begin{equation}
\label{192} P [\tau\notin G] \leq3 \delta.
\end{equation}
For $(\omega,t) \in G$ such that $\frac{k-1}{2^m} < t \leq\frac
{k}{2^m}$, we have
\[
\bigl|X_{t-}(\omega) - X_{({k-1})/{2^m}} (\omega) \bigr| \leq\varepsilon
\]
so that by \eqref{192} we get
\[
P \bigl[| X_{\tau-} - X_{\tau_m-2^{-m}}| > \varepsilon \bigr] < 3 \delta,
\]
which shows \eqref{P3a}.
\end{pf}
\begin{pf*}{Proof of Proposition~\ref{propti}}
Fix $\varepsilon> 0$, and apply Lemma~\ref{lW2} to find $m\in
\mathbb
{N}$ such that
%
%e6.12 #&#
\begin{equation}
\label{193}
P \bigl[| \widetilde{X}_{\tau_m-2^{-m}} - \widetilde{X}_{\tau-}
| > \varepsilon \bigr] < \varepsilon,
\end{equation}
for each $\widetilde{X} \in\mathfrak{X}$. As $(X^n_q)^\infty_{n=1}$
converges to $X_q$ in probability, for every rational number $q \in
\mathbb{Q} \cap[0,1]$ we have
\[
P \biggl[\max_{0 \leq k \leq2^m} \bigl| X^n_{{k}/{2^m}} -
X_{{k}/{2^m}} \bigr| > \varepsilon \biggr] < \varepsilon,
\]
for all $n\geq N(\varepsilon)$. We then may apply \eqref{193} to $X^n$
and $X$ to conclude that
\[
P\Bigl[\bigl| X^n_{\tau-} - X_{\tau-} \bigr| > 3 \varepsilon
\Bigr] < 3\varepsilon.
\]
\upqed\end{pf*}
With Proposition~\ref{propti} we have now everything in place to prove
Theorem~\ref{t3}.
\begin{pf*}{Proof of Theorem~\ref{t3}}
The existence of the optional strong supermartingale $X^{(1)}$ is the
assertion of Theorem~\ref{t1}. To obtain the predictable strong
supermartingale $X^{(0)}$, we observe that, since $\widetilde{X}^n$
and $X^{(1)}$
are l\`adl\`ag, the optional set
\[
F:=\bigcup^\infty_{n=1} \bigl\{\widetilde{X}^n \neq
\widetilde{X}^n_-\bigr\} \cup\bigl\{X^{(1)} \neq
X^{(1)}_-\bigr\}
\]
has at most countably many sections, and therefore there exists by
Theorem~117 in Appendix IV of \cite{DM78} a countable number of $[0,1]
\cup\{\infty\}$-valued stopping times $(\sigma_m)^{\infty}_{m=1}$ with
disjoint graphs such that $F=\bigcup_{m=1}^\infty\llbracket\sigma
_m\rrbracket$. By Theorem IV.81.c) in \cite{DM78} we can decompose each
stopping time $\sigma_m$ into an accessible stopping time $\sigma^A_m$
and a totally inaccessible stopping time $\sigma^I_m$ such that
$\sigma
_m=\sigma^A_m \wedge\sigma^I_m$. Again combining Koml\'os's lemma with
a diagonalization procedure we obtain a sequence of convex combinations
$\widetilde{X}^n\in\operatorname{conv}(X^n, X^{n+1}, \ldots)$ such
that $\widetilde{X}^n_\tau\stackrel
{P}{\longrightarrow} X^{(1)}_\tau$ for all $[0,1]$-valued stopping
times $\tau$ as well as
\[
\widetilde{X}^n_{\tau_m-} \mathop{-\!\!\!\!-\!\!\!\longrightarrow}^{P\mbox{-}\mathrm{a.s.}} Y^{(0)}_m,
\qquad\mbox{as }n\to\infty,
\]
for all stopping times $\tau_m:=\sigma^A_m\wedge1$ and suitable
nonnegative random variables $Y^{(0)}_m$ for $m\in\mathbb{N}$. Now we
can define $X^{(0)}$ by
\[
X^{(0)}_t(\omega) =
\cases{
Y^{(0)}_m(\omega), &\quad $t = \sigma_m^A(
\omega)$ and $m\in \mathbb{N}$,\vspace*{3pt}
\cr
X^{(1)}_{t-}(
\omega)=X^{(1)}_t(\omega), & \quad $\mbox{else}$.}
\]
For all $[0,1]$-valued stopping times $\tau$, we then have convergence
\eqref{eqt32}, that is,
\begin{eqnarray*}
\widetilde{X}^n_{\tau-} (\omega) &=& \widetilde{X}^n_{\tau}
(\omega ) \mathbh{1}_F \bigl(\omega, \tau (\omega) \bigr) + \sum
^{\infty}_{m=1} \widetilde{X}^n_{\tau_m^-}
\mathbh{1}_{\{\sigma
^A_m=\tau\}} + \sum^{\infty}_{m=1}
\widetilde{X}^n_{\sigma^I_m-} \mathbh{1}_{\{
\sigma^I_m=\tau\}}
\\
&\stackrel{P}\longrightarrow & X_{\tau}^{(0)}(\omega)
\mathbh{1} _F \bigl(\omega, \tau, (\omega) \bigr) + \sum
^\infty_{m=1} Y^{(0)}_m
\mathbh{1}_{\{
\sigma
^A_m=\tau\}} + \sum^\infty_{m=1}
X_{\sigma^I_m-}^{(1)} \mathbh {1}_{\{
\sigma_m^I=\tau\}},
\end{eqnarray*}
since $\widetilde{X}^n=\widetilde{X}^n_-$ for all $n\in\mathbb{N}$
on $F$ and $\widetilde{X}^n_{\sigma
-} \mathbh{1}_{\{\sigma=\tau\}} \stackrel{P}{\longrightarrow} X_{\sigma-}
\mathbh{1}_{\{\sigma=\tau\}}$ for all $[0,1]$-valued totally inaccessible
stopping times $\tau$ by Proposition~\ref{propti}. As all stopping
times $\sigma_m^A$ are accessible and each $Y_m^{(0)}$ is $\mathcal
{F}_{\tau_{m}-}$-measurable, we have that $X^{(0)}$ is an accessible
process such that $X^{(0)}_\tau\mathbh{1}_{\{\tau< \infty\}}$ is
$\mathcal{F}_{\tau-}$-measurable for every stopping time $\tau$.
Therefore $X^{(0)}$ is by Theorem~3.20 in~\cite{D72} even predictable.
By Remark~5.(c) in Appendix~I of \cite{DM82} the left limit\vspace*{1pt} process
$\widetilde{X}
^n_-$ of each optional strong supermartingale $\widetilde{X}^n$ is a
predictable
strong supermartingale satisfying
\[
\widetilde{X}^n_{\tau-} \geq E\bigl[\widetilde{X}^n_\tau|
\mathcal {F}_{\tau-}\bigr]
\]
for all $[0,1]$-valued predictable stopping times. Therefore\vspace*{1pt} the
predictable strong supermartingale property [part (3) of Definition~\ref
{defpred}] and $X^{(0)}_{\tau} \geq E[X^{(1)}_{\tau} | \mathcal
{F}_{\tau-}]$
follow immediately from \eqref{eqt31} and \eqref{eqt32} by Fatou's
lemma. To see $X^{(1)}_{\tau-}\geq X^{(0)}_{\tau}$, let $(\tau
_m)_{m=1}^\infty$ be a foretelling sequence of stopping times for the
predictable stopping time $\tau$. Then we have
\[
\widetilde{X}^{n}_{\tau_m} \geq E\bigl[\widetilde{X}^{n}_{\tau_{m+k}}
| \mathcal{F}_{\tau_m}\bigr]
\]
for all $n,m,k\in\mathbb{N}$. Applying Fatou's lemma we then obtain
\[
\widetilde{X}^{n}_{\tau_m} \geq E\bigl[\widetilde{X}^{n}_{\tau-}
| \mathcal{F}_{\tau_m}\bigr]
\]
by sending $k\to\infty$,
\[
X^{(1)}_{\tau_m} \geq E\bigl[X^{(0)}_{\tau-} |
\mathcal{F}_{\tau_m}\bigr]
\]
by sending also $n\to\infty$ and finally
$X^{(1)}_{\tau-}\geq X^{(0)}_{\tau}$ by sending $m\to\infty$.
\end{pf*}
%
%s7 #&#
\section{Proof of Proposition \texorpdfstring{\protect\ref{pSI}}{2.12}}\label{sec7}
One application of Theorem~\ref{t3} is a convergence result for
stochastic integrals of predictable integrands of finite variation with
respect to nonnegative optional strong supermartingales.

Fix a nonnegative optional strong supermartingale $X \in\mathfrak
{X}$, and let $\varphi= (\varphi_t)_{0 \leq t \leq1}$ be a
predictable process of finite variation, so that it has l\`adl\`ag
paths. We then define\vspace*{-2pt}
%
%e7.1 #&#
\begin{eqnarray}
\int^t_0 X_u(
\omega) \,d\varphi_u(\omega) &:=&  \int^t_0
X_u(\omega) \,d\varphi^c_u(\omega) + \sum
_{0 <u \leq t} X_{u-}(\omega) \Delta
\varphi_u(\omega)
\nonumber
\\[-10pt]
\label{defSI1}\\[-10pt]
\nonumber
&&{}+ \sum_{0 \leq u
< t}
X_u(\omega) \Delta_+\varphi_u(\omega)
\end{eqnarray}
for all $t\in[0,1]$, which is $P$-a.s. pathwise well defined, as $X$ is
l\'adl\'ag and $\varphi$ of finite variation. Here the integral $\int^t_0 X_u (\omega) \,d\varphi^c_u(\omega)$ with respect to the
continuous part
$\varphi^c$ [see \eqref{defcont}] can be defined as a pathwise
Riemann--Stieltjes integral or a pathwise Lebesgue--Stieltjes integral,
as both integrals coincide.

To ensure the integration  by parts formula
%
%e7.2 #&#
\begin{equation}\label{SIPI}
\hspace*{18pt}\varphi_t(\omega)X_t(\omega) - \varphi_0(
\omega) X_0 (\omega) = \int^t_0
\varphi _u(\omega) \,dX_u(\omega) + \int
^t_0 X_u (\omega)\,d
\varphi_u(\omega ),
\end{equation}
we define the stochastic integral $\varphi\bolds{\cdot}X_t:=\int^t_0 \varphi
_u \,dX_u$ by
%
%e7.3 #&#
\begin{eqnarray}
\qquad\int^t_0 \varphi_u
(\omega) \,dX_u (\omega)& :=& \int^t_0
\varphi ^c_u (\omega) \,dX_u (\omega)+ \sum
_{0 <u \leq t} \Delta\varphi_u(\omega)
\bigl(X_t(\omega) - X_{u-}(\omega) \bigr)
\nonumber
\\[-9pt]
\label{defSI2}\\[-9pt]
\nonumber
&&{}+ \sum_{0 \leq u < t} \Delta_+\varphi_u(
\omega) \bigl(X_t(\omega ) - X_{u}(\omega ) \bigr)
\end{eqnarray}
for $t \in[0,1]$ that is again pathwise well defined. The integral
$\int^t_0 \varphi^c_u (\omega) \,dX_u (\omega)$ can again be defined
as a
pathwise Riemann--Stieltjes integral or a pathwise Lebesgue--Stieltjes
integral. If $X=(X_t)_{0 \leq t \leq1}$ is a semimartingale, the
definition of $(\int^t_0 \varphi_u \,dX_u)_{0\leq t\leq1}$ via \eqref{defSI2} coincides with the classical stochastic integral.

We first derive an auxiliary\vspace*{-2pt} result.

%le7.1 #&#
\begin{lemma}\label{lSI}
Let $(X^n)^\infty_{n=1}$, $X^{(0)}$ and $X^{(1)}$ be l\`adl\`ag
stochastic processes such\vspace*{-4pt} that:
\begin{longlist}[(ii)]
\item[(i)] $X^n_\tau\stackrel{P}{\longrightarrow}
X^{(1)}_\tau$
and $X^n_{\tau-} \stackrel{P}{\longrightarrow} X^{(0)}_\tau$ for all
$[0,1]$-valued stopping times $\tau$;
\item[(ii)] for all $\varepsilon> 0$ and $\delta> 0$, there are
constants $C_1(\delta) > 0$ and $C_2(\varepsilon, \delta) > 0$ such that
%
%e7.4 #&#
%e7.5 #&#
\begin{eqnarray}
\label{lSIa}\sup_{X\in\mathcal{X}^0} P\Bigl[\sup_{0 \leq s \leq1} |
X_s | > C_1 (\delta )\Bigr] &\leq & \delta,
\\[-2pt]
\label{lSIb}
\sup_{X\in\mathcal{X}^1} P\bigl[M_\varepsilon(X) > C_2 (
\varepsilon, \delta )\bigr] &\leq & \delta,
\end{eqnarray}
where $\mathcal{X}^0=\{X^{(0)}, X^{(1)}, X^n, X^n_- \mbox{ for }n \in
\mathbb{N}\}$, $\mathcal{X}^1=\{X^{(1)}, X^n\mbox{ for }n \in
\mathbb{N}\}$ and
\[
M_\varepsilon(X):=\sup \Bigl\{m \in\mathbb{N}  \Big| \bigl|X_{t_i} (
\omega) - X_{t_{i-1}} (\omega) \bigr| > \varepsilon\mbox{ for }0 \leq
t_0 < t_1<\cdots  < t_m \leq 1 \Bigr\}
\]
for $X \in\mathcal{X}^1$.
\end{longlist}
Then we have, for all predictable processes $\varphi=(\varphi_t)_{0
\leq t \leq1}$ of finite variation, that:
\renewcommand{\theequation}{\arabic{equation}}
\setcounter{equation}{0}
\begin{eqnarray}
\label{1}
\int^\tau_0 X^n_u \,d\varphi_u  & \stackrel
{P}{\longrightarrow} &  \int^\tau_0 X^{(1)}_u \,d\varphi^c_u + \sum_{0
< u
\leq\tau} X_u^{(0)} \Delta\varphi_u + \sum_{0 \leq u < \tau}
X_u^{(1)} \Delta_+\varphi_u;
\\
\int^\tau_0 \varphi_u \,dX^n_u  & \stackrel
{P}{\longrightarrow} & \int^\tau_0 \varphi^c_u \,dX^{(1)}_u + \sum_{0
< u
\leq\tau} \Delta\varphi_u \bigl(X^{(1)}_\tau- X^{(0)}_u\bigr)
\nonumber
\\[-8pt]
\label{2}
\\[-8pt]
\nonumber
&&{}+ \sum_{0
\leq u
< \tau} \Delta_+  \varphi_u \bigl(X^{(1)}_\tau- X^{(1)}_u\bigr)
\end{eqnarray}
for all $[0,1]$-valued stopping times $\tau$. Convergence (1) is even
uniformly in probability.
\end{lemma}

\renewcommand{\theequation}{\thesection.\arabic{equation}}
\setcounter{equation}{5}
\begin{pf}
\eqref{1} We first show that
%
%e7.6 #&#
\begin{equation}\label{eqcl1}
\sup_{0\leq t\leq1}\biggl\llvert \sum_{0 < u \leq t}
X^n_{u-} \Delta \varphi_u - \sum
_{0 < u \leq t} X^{(0)}_{u-} \Delta
\varphi_u\biggr\rrvert \stackrel {P} {\longrightarrow}0, \qquad \mbox{as }n
\to\infty,
\end{equation}
that is, uniformly in probability. The proof of the convergence
\[
\sup_{0\leq t\leq1}\biggl\llvert \sum_{0 < u \leq t}
X^n_u \Delta_+ \varphi_u - \sum
_{0 < u \leq t} X^{(1)}_{u-} \Delta_+
\varphi_u\biggr\rrvert \stackrel {P} {\longrightarrow}0, \qquad \mbox{as }n
\to\infty,
\]
is completely analog and therefore omitted.

Since $\varphi$ is predictable and of finite variation and hence l\`
adl\`ag, there exists a sequence $(\tau_m)_{m=1}^\infty$ of
$[0,1]\cup\{
\infty\}$-valued stopping times exhausting the jumps of~$\varphi$.
Using the stopping times $(\tau_m)_{m=1}^\infty$ we can write
\[
\sum_{0 < u \leq t} X_u \Delta
\varphi_u = \sum^{\infty}_{m=1}
X_{\tau
_m} \Delta\varphi_{\tau_m} \mathbh{1}_{\{\tau_m \leq t\}}
\]
for all $X\in\mathcal{X}^0$ and estimate
\begin{eqnarray}
&&\hspace*{8pt} \sup_{0\leq t\leq1}\Biggl\llvert \sum
^\infty_{m=1} X^n_{\tau_m-} \Delta
\varphi_{\tau_m} \mathbh{1}_{\{\tau_m \leq t\}} - \sum
^\infty_{m=1} X^{(0)}_{\tau_m} \Delta
\varphi_{\tau_m} \mathbh{1}_{\{\tau_m
\leq t\}
} \Biggr\rrvert
\nonumber
\\[-8pt]
\label{eqSI1}\\[-8pt]
\nonumber
&& \hspace*{8pt}\qquad \leq\sum^N_{m=1} \bigl| X^n_{\tau_m-}
- X^{(0)}_{\tau_m} \bigr| \bigl| \Delta \varphi _{\tau_m} \bigr| + \sup
_{m \in\mathbb{N}} \bigl| X^n_{\tau_m-} - X^{(0)}_{\tau
_m}
\bigr| \sum^\infty_{m=N+1} | \Delta
\varphi_{\tau_m} |.
\nonumber
\end{eqnarray}
Combining \eqref{eqSI1} with the fact that $\varphi$ is of finite
variation we obtain \eqref{eqcl1}, as
\[
\sup_{m \in\mathbb{N}} \bigl| X^n_{\tau_m-} -
X^{(0)}_{\tau_m} \bigr| \sum^\infty
_{m=N+1} | \Delta\varphi_{\tau_m} |\stackrel{P} {\longrightarrow
}0,\qquad \mbox{as }N\to\infty,
\]
by \eqref{lSIa} and $\sum^N_{m=1} | X^n_{\tau_m-} - X^{(0)}_{\tau_m}
| | \Delta\varphi_{\tau_m} |\stackrel{P}{\longrightarrow}0$, as
$n\to
\infty$, for each $N$ by assumption (i).

The key observation for the proof of the convergence
%
%e7.7 #&#
\begin{equation}
\label{eqcl3}
\sup_{0\leq t\leq1}\biggl\llvert \int
^t_0 X^n_u \,d
\varphi^c_u - \int^t_0
X^{(1)}_u \,d \varphi^c_u\biggr
\rrvert \stackrel{P} {\longrightarrow}0, \qquad \mbox{as }n\to\infty,
\end{equation}
is that we can use assumption (ii) to approximate the stochastic
Riemann--Stieltjes integrals by Riemann sums in probability uniformly
for all $X \in\mathcal{X}^1$, as either the integrator or the
integrand moves very little. Indeed, for $\varepsilon>0$ and
$c_1,c_2>0$, we
have that
\begin{eqnarray*}
&& \sup_{0\leq t\leq1}
\Biggl\llvert \int^t_0
X_u\,d\varphi^c_u - \sum
^N_{m=1} X_{\sigma_{m-1}} \bigl(
\varphi^c_{\sigma_m\wedge t} - \varphi ^c_{\sigma
_{m-1}\wedge t}
\bigr) \Biggr\rrvert
\\[2pt]
&& \qquad \leq\sum^N_{m=1} \sup
_{u \in[\sigma_{m-1}, \sigma_m]} | X_u - X_{\sigma_{m-1}} | \bigl(\bigl|
\varphi^c \bigr| _{\sigma_m} - \bigl| \varphi ^c\bigr|_{\sigma
_{m-1}}
\bigr)
\\[2pt]
&&\qquad
\leq c_2 2c_1 \frac{\varepsilon}{4c_1c_2} +
\frac
{\varepsilon}{2c_1} c_1 = \varepsilon
\end{eqnarray*}
on $\{|\varphi|_1\leq c_1\} \cap\{X^*_1 \leq c_1\} \cap\{M_{{\varepsilon}/({2c_1})} (X) \leq c_2\}$, where the stopping times
$(\sigma
_m)^\infty_{m=0}$ are given by $\sigma_0=0$ and
\[
\sigma_m:= \inf \biggl\{t > \sigma_{m-1} \bigg| \bigl|
\varphi^c\bigr|_t - \bigl|\varphi ^c\bigr|_{\sigma_{m-1}} >
\frac{\varepsilon}{4c_1 c_2} \biggr\} \wedge1
\]
and $N=\frac{4c_1 c_2}{\varepsilon}$. Choosing $c_1,c_2>0$ and hence
$N$ sufficiently large we therefore obtain\vspace*{-3pt}
\[
\sup_{X \in\mathcal{X}^1}P \Biggl(\sup_{0\leq t\leq1}\Biggl\llvert
\int^t_0 X_n\,d\varphi^c_u
- \sum^N_{m=1} X_{\sigma_{m-1}} \bigl(
\varphi ^c_{\sigma
_m\wedge t} - \varphi^c_{\sigma_{m-1}\wedge t}
\bigr) \Biggr\rrvert >\varepsilon \Biggr)<\delta
\]
for any $\delta>0$ by assumption (ii). Combing this with the estimate
\begin{eqnarray*}
&& \sup_{0\leq t\leq1}\biggl\llvert \int^t_0
X^n_u \,d \varphi^c_u - \int
^t_0 X^{(1)}_u \,d
\varphi^c_u \biggr\rrvert \\[2pt]
&&\qquad\leq  \sup_{0\leq t\leq1}
\Biggl\llvert \int^t_0 X^n_u
\,d \varphi^c_u - \sum^N_{m=1}
X^n_{\sigma_{m-1}} \bigl(\varphi ^c_{\sigma_m\wedge t} -
\varphi^c_{\sigma_{m-1} \wedge t} \bigr)\Biggr\rrvert
\\[2pt]
&&\qquad\quad{}+\sum^N_{m=1} \bigl| X^n_{\sigma_{m-1}}
- X^{(1)}_{\sigma_{m-1}} \bigr| \bigl(\bigl|\varphi^c
\bigr|_{\sigma_m} - \bigl|\varphi^c \bigr|_{\sigma_{m-1}} \bigr)
\\[2pt]
&&\qquad\quad{}+\sup_{0\leq t\leq1}\Biggl\llvert \int^t_0
X^{(1)}_u \,d \varphi^c_u - \sum
^N_{m=1} X^{(1)}_{\sigma_{m-1}}
\bigl(\varphi^c_{\sigma_m\wedge t} - \varphi ^c_{\sigma_{m-1}\wedge t}
\bigr)\Biggr\rrvert
\end{eqnarray*}
then implies \eqref{eqcl3}, as
\[
\max_{m=0, \ldots, N-1} \bigl| X^n_{\sigma_m} -
X^{(1)}_{\sigma
_m}\bigr|\stackrel {P} {\longrightarrow}0,\qquad \mbox{as }n\to
\infty,
\]
for each fixed $N$ by assumption (i).

\eqref{2} As $X^n_\tau\varphi_\tau\stackrel{P}{\longrightarrow}
X^{(1)}_\tau\varphi_\tau$ for all $[0,1]$-valued stopping times, this
assertion follows immediately from  \eqref{1} and the integration by
parts formula \eqref{SIPI}.
% taking into account that
%$$\sum_{0 < u \leq\tau} X^n_{u-} \Delta\varphi_u \stackrel{P}{
%\longrightarrow} \sum_{0 < u \leq\tau} X^{(0)}_u \Delta\varphi_u$$
%and therefore
%\begin{multline*}
%\sum_{0 < u \leq\tau} \Delta\varphi_u (X^n_\tau- X^n_{\tau-}) = X^n_
%\tau\left(\sum_{0 < u \leq\tau} \Delta\varphi_u\right) - \sum_{0 < u
%\leq\tau} X^n_{u-} \Delta\varphi_u\\
%\stackrel{P}{\longrightarrow} X^{(1)}_\tau\left(\sum_{0 < u \leq
%\tau} \Delta\varphi_u\right) + \sum_{0 < u \leq\tau} X^{(0)}_u
%\Delta\varphi_u = \sum_{0 < u \leq\tau} \Delta\varphi_u (X^{(1)}_
%\tau- X^{(0)}_u).
%\end{multline*}
\end{pf}
Combining the previous lemma with Lemma~\ref{lW1} allows us now to
complete the proof of Proposition~\ref{pSI}.
\begin{pf*}{Proof of  Proposition~\ref{pSI}}
Part (1) is Theorem~\ref{t3}, and part (2) follows from Lemma~\ref
{lSI} as soon as we have shown that its assumptions are satisfied.
Assumption (i) is \eqref{1} and for the set $\mathcal{X}^1$ assumption
(ii) can be derived from Lemma~\ref{lW1}. Therefore it only remains to
show \eqref{lSIa} for $X^{(0)}$ and $X^n_-$ for $n \in\mathbb{N}$.
For the left limits \eqref{lSIa} follows from the validity of the
latter for the processes $X^n$ for $n \in\mathbb{N}$ and for the
predictable strong supermartingale $X^{(0)}$ from (3.1) in Appendix~I
of \cite{DM82}.
\end{pf*}

\section*{Acknowledgments}

We would like to thank an anonymous referee for
careful reading of the paper and pertinent remarks.

% imsref loaded by daiva.urboniene, 2014-10-29 14:00:10

%\begin{appendix}
%\section{}
%\end{appendix}

%\section*{Acknowledgments}

%\begin{supplement}[id=suppA]
%\slink[doi]{10.1214/00-AOPXXXXSUPP} %[doi,text={...}] - jei reikia
%suskaldyti doi
%\sfilename{aopXXXX\_supp.pdf}
%\end{supplement}

%\begin{thebibliography}{99}
%\bibitem[\protect\citeauthoryear{}{}]{r1}
%\bibitem{r1}
%\end{thebibliography}

\printaddresses
\end{document}